\documentclass{amsart}
\usepackage{tikz,tikz-cd}
\usepackage{url}
\usepackage{graphicx}
\usepackage{graphicx}
\usepackage{color}
\usepackage{times}
\usepackage{cite}
\usepackage{enumerate,latexsym, enumitem}
\usepackage{latexsym}
\usepackage{amsmath,amssymb}% CUIDADO: Si uso estos packages, $\widetilde{}$ no funciona a veces
\usepackage{graphicx}
\usepackage{amsthm}
\usepackage{verbatim}
\newtheorem{thm}{Theorem}[section]
\newtheorem*{theorem*}{Theorem}
\newtheorem*{acknowledgement*}{Acknowledgment}
\newtheorem{cor}[thm]{Corollary}

\newtheorem{lem}[thm]{Lemma}
\newtheorem{prop}[thm]{Proposition}
\theoremstyle{definition}

\theoremstyle{remark}
\newtheorem{rem}[thm]{Remark}
\newtheorem{ques}[thm]{Question}
\newtheorem{conj}[thm]{Conjecture}
\numberwithin{equation}{section}

\newcommand{\set}[1]{\left\{#1\right\}}
\newcommand{\Real}{\mathbb R}

\newcommand{\vte}{virtual entropy}
\newcommand{\WW}[2]{W^{#1}_{#2}}

\title[Conformal Volume and Entropy of Self-Shrinkers]{Li-Yau Conformal Volume and Colding-Minicozzi Entropy of Self-Shrinkers}
\author{Jacob Bernstein}
\address{Department of Mathematics, Johns Hopkins University, 3400 N. Charles Street, Baltimore, MD 21218}
\email{bernstein@math.jhu.edu}
\thanks{The author was partially supported by the NSF Grant  DMS-2203132.  The author would like to thank Jeffrey Case for a stimulating discussion that lead to this work and Letian Chen for helpful comments on an earlier draft.}

\begin{document}

\begin{abstract}
We show the (normalized) Li-Yau conformal volume of a self-shrinker of mean curvature flow in Euclidean space bounds its Colding-Minicozzi entropy from below.  This bound is independent of codimension and sharp on planes.  As an application we verify a conjecture of Colding-Minicozzi about the entropy of closed self-shrinkers of arbitrary codimension for self-shrinkers that are topologically two-dimensional real projective planes.  As part of the proof we introduce two auxiliary functionals which we call \emph{stable conformal volume} and \emph{\vte} which should be of independent interest. 
\end{abstract}
\maketitle
\section{Introduction}
In \cite{liNewConformalInvariant1982a}, Li-Yau introduced a quantity associated to branched conformal immersions into $\mathbb{S}^N$ they called the ($N$-)\emph{conformal volume} -- this quantity is also called \emph{visual volume} by Gromov \cite{gromovFillingRiemannianManifolds1983}.  Among other applications, it has proven to be of significant importance both in the study of the Willmore energy and spectral properties of surfaces.  There is a vast literature on this quantity -- we refer in particular to \cite{bryantSurfacesConformalGeometry1988} for helpful background discussion.  In this paper, we relate the conformal volume to another important geometric quantity introduced in \cite{Coldinga} called  \emph{Colding-Minicozzi entropy} or \emph{entropy} for short, which has proven important in the study of mean curvature flow.

Consider the inverse stereographic projection map
\begin{equation}\label{SteroEqn}
\mathcal{S}:\Real^N \to \mathbb{S}^{N}\subset \Real^{N+1}, \mathbf{x}\mapsto \left( \frac{\mathbf{x}}{1+ \frac{1}{4}|\mathbf{x}|^2}, \frac{1-\frac{1}{4}|\mathbf{x}|^2}{1+\frac{1}{4}|\mathbf{x}|^2}\right).
\end{equation}
For an $n$-dimensional submanifold, $\Sigma \subset \Real^N$ and embedding $j:M^n \to \Sigma\subset \Real^N$ define
$$
\lambda_{LY}^n[\Sigma]=V_{c}(N, \mathcal{S}\circ j)
$$
where $V_{c}(N, \mathcal{S}\circ j)$ is the $N$-conformal volume of the embedding $\mathcal{S}\circ j:M \to \mathcal{S}(\Sigma)\subset \mathbb{S}^N$ introduced in \cite{liNewConformalInvariant1982a}. A straightforward computation yields, 
$$
\lambda_{LY}^n[\Sigma]=\sup_{\rho>0, \mathbf{x}_0\in \Real^N} \int_{\Sigma} \WW{ n}{\rho}(\mathbf{x}-\mathbf{x}_0) d\mathcal{H}^n(\mathbf{x}).
$$
where the weights $\WW{ n}{\rho}$ are given explicitly by
$$
\WW{ n}{\rho}(\mathbf{x})=  \frac{\rho^n}{(1+\frac{1}{4}\rho^2|\mathbf{x}|^2)^{n}}.
$$
These are the volume element on $\Sigma$ induced by homothetic scalings of $\mathcal{S}^* g_{\mathbb{S}}$, i.e.,  the pull-back of the round metric on the sphere. Denote the \emph{normalized} conformal volume by
$$
\bar{\lambda}_{LY}^n[\Sigma]=|\mathbb{S}^n|^{-1} \lambda_{LY}^n[\Sigma].
$$
When $\Sigma$ is an affine $n$-plane, $\Real^n$ or round $n$-sphere, $\mathbb{S}^n$,  $\bar{\lambda}_{LY}^n[\Sigma]=1$.

Following \cite{Coldinga}, for a $n$-dimensional submanifold $\Sigma\subset \Real^N$, the entropy of $\Sigma$ is the value
$$
\lambda_{CM}^n[\Sigma]=\sup_{\rho>0, \mathbf{x}_0\in \Real^N} (4\pi \rho)^{-\frac{n}{2}} \int_{\Sigma} e^{-\frac{|\mathbf{x}-\mathbf{x}_0|^2}{4\rho}} d\mathcal{H}^n.
$$
This quantity is monotone non-increasing along mean curvature flows and has found many applications both in the study of mean curvature flow and as a measure of complexity, especially of closed hypersurfaces, e.g.,  \cite{BernsteinWang1, BernsteinWang2, BernsteinWang3,BWIsotopy, BWDensity, CCMS, CMS, JZhu}. For affine $n$-planes, $\Real^n$, the value is one, but is strictly larger for all other proper submanifolds -- see \cite{chenRigidityStabilitySubmanifolds2021, BernsteinRigidity}.
While there has been extensive study of the entropy of hypersurfaces, comparatively little is known about it in higher codimension -- see however recent work of Colding-Minicozzi \cite{coldingComplexityParabolicSystems2020, coldingEntropyCodimensionBounds2019}.

One readily observes -- see Proposition \ref{ConfVolCMentProp} --  that
$$
(4\pi \rho)^{-\frac{n}{2}} e^{-\frac{|\mathbf{x}|^2}{4\rho}}\leq  \left( \frac{n}{4\pi}\right)^{\frac{n}{2}} \WW{ n}{ (n\rho)^{-\frac{1}{2}}}(\mathbf{x})
$$
which leads to the non-sharp bound
$$
\lambda_{CM}^n[\Sigma]\leq \left( \frac{n}{4\pi}\right)^{\frac{n}{2}} |\mathbb{S}^n| \bar{\lambda}_{LY}^n[\Sigma].
$$
Note that asymptotically the constant behaves like
$$
\left( \frac{n}{4\pi}\right)^{\frac{n}{2}} |\mathbb{S}^n|=\sqrt{2} \left( \frac{e}{2}\right)^{\frac{n}{2}}+l.o.t, n\to \infty.
$$
which grows rapidly as $n$ grows.

We introduce two auxiliary measures of complexity of submanifolds and then relate the new quantities with the conformal volume and entropy. These relationships are particularly striking when the submanifold, $\Sigma\subset \Real^N$, is a  \emph{self-shrinker}.  That is, satisfies
$$
\mathbf{H}_\Sigma+\frac{\mathbf{x}^\perp}{2}=\mathbf{0}
$$
where $\mathbf{H}_\Sigma$ is the mean curvature vector and $\mathbf{x}^\perp$ is the normal component of the position vector.  Such submanifolds are important in the study of mean curvature flow, i.e., in the geometric evolution equation
$$
t\in I\mapsto \Sigma_t \subset \Real^{N}, \left(\frac{\partial \mathbf{x}}{\partial t}\right)^\perp = \mathbf{H}_{\Sigma_t}.
$$
Here $I$ is an interval of time and $t\mapsto \Sigma_t$ is a smooth evolution of $n$-dimensional submanifolds.
This is because for a self-shrinker $\Sigma$, 
$$t\in (-\infty, 0)\mapsto \sqrt{-t} \Sigma$$
 is an evolution by mean curvature flow that provides a model of singularity formation of the flow. As $\lambda_{CM}^n$ is invariant under rescaling, it is constant along this flow. 
 \begin{thm}\label{LYCMIneqThm}
 	Let $\Sigma\subset \Real^N$ be a $n$-dimensional properly embedded submanifold that is, up to translation and dilation, a self-shrinker.  One has
 	$$
1\leq \bar{\lambda}_{LY}^n[\Sigma]\leq \lambda_{CM}^n[\Sigma].
 	$$
 \end{thm}  The  inequality of Theorem \ref{LYCMIneqThm} is sharp in the sense that equality holds throughout on affine $n$-planes and by \cite{chenRigidityStabilitySubmanifolds2021} this double equality characterizes these planes.  When $n\geq 2$ planes and round spheres are characterized by equality in the first inequality by \cite{bryantSurfacesConformalGeometry1988}, though for round spheres the second inequality is strict.  While they are not self-shrinkers, certain double logarithmic spirals provide examples with equality in the second inequality but not the first -- see Appendix \ref{LoxodromeSec}.

In \cite[pg. 119]{coldingComplexityParabolicSystems2020}, Colding-Minicozzi made the following conjecture:
\begin{conj} \label{CMConj} If $\Sigma\subset \Real^N$ is a closed $n$-dimensional self-shrinker, then 
	$$\sqrt{2}=\lim_{n\to \infty} \left(\frac{2n}{4\pi e}\right)^{\frac{n}{2}}|\mathbb{S}^n|<\left(\frac{2n}{4\pi e}\right)^{\frac{n}{2}}|\mathbb{S}^n|=\lambda_{CM}^n[\mathbb{S}^n]\leq \lambda_{CM}^n[\Sigma].$$
\end{conj}
When $n=N+1$, i.e., $\Sigma$ is a hypersurface, this was shown in \cite{CIMW} and when $n=1$ it follows from the classification of self-shrinking curves \cite[Lemma 6.16]{coldingComplexityParabolicSystems2020}. 
It was shown in \cite{BernsteinWang2}   that the only non-flat proper self-shrinking surfaces in $\Real^3$ with entropy below that of the shrinking cylinder, i.e., less than $\lambda_{CM}^2[\mathbb{S}^1\times \Real]=\lambda_{CM}^1[\mathbb{S}^1]\approx 1.52$, are the flat plane which has entropy 1, the shrinking sphere which has entropy $\lambda_{CM}^2[\mathbb{S}^2]\approx 1.47$ and the shrinking cylinder itself. A strengthening of Conjecture \eqref{CMConj} is that this phenomena continues to hold for surfaces of any codimension  \cite[Conjecture 8.1]{coldingComplexityParabolicSystems2020}.

In  \cite{liNewConformalInvariant1982a},  Li-Yau observed that two-dimensional real projective planes have conformal volume at least $6\pi$.  Thus, Theorem \ref{LYCMIneqThm} gives is a partial confirmation of Conjecture \ref{CMConj}:
\begin{cor}
	If $\Sigma\subset \Real^N$ is a self-shrinker that is topologically $\mathbb{RP}^2$, then
	$$
1.47\approx	\lambda_{CM}^2[\mathbb{S}^2]<\frac{3}{2} \leq \lambda_{CM}^2[\Sigma].
	$$
\end{cor}
There is a minimal $\mathbb{RP}^2$ embedded into $\mathbb{S}^4\subset \Real^5$, which we denote by $\Sigma_{\mathbb{RP}^2}$ and is sometimes called the Veronese surface, which has conformal volume $6\pi$.  One may check that $2 \Sigma_{\mathbb{RP}^2}$ is a self-shrinker and compute that
$$
\bar{\lambda}_{LY}^2[\Sigma_{\mathbb{RP}^2}]=\frac{3}{2}< 2.207\approx\frac{6}{e}=\frac{3}{2} \lambda_{CM}^2[\mathbb{S}^2]= \lambda^2_{CM}[\Sigma_{\mathbb{RP}^2}].
$$

In the case of tori, the conformal area has been estimated in terms of the complex structure by Montiel and Ros \cite{montielMinimalImmersionsSurfaces1986}.  Moreover, they gave exact values for a large class of tori; this was refined by Bryant in \cite{bryantConformalVolume2tori2015}. 
\begin{cor}\label{ToriEntropyCor}
	Let $\Sigma\subset \Real^N$ be a self-shrinking torus conformally parametrized by the lattice $\mathbb{C}/\Lambda_\omega$ where $\omega=x+i y$ and $0\leq x\leq \frac{1}{2}$, $\sqrt{1-x^2}\leq y$.  
	
	One has 
	$$
 \frac{\pi y}{x^2+y^2-x+1}\leq \bar{\lambda}_{LY}^2[\Sigma]\leq 	\lambda_{CM}^2[\Sigma].
	$$
\end{cor}
\begin{rem}
	This bound comes from the bound established in \cite{montielMinimalImmersionsSurfaces1986} on (unnormalized) conformal volume of the lattice
	$$
 \frac{4\pi^2 y}{x^2+y^2-x+1}\leq 	V_c(\mathbb{C}/\Lambda_\tau, [g_{\Real}]).
	$$
	This inequality becomes strict for large $y$.  However, as conjectured and partially shown in \cite{montielMinimalImmersionsSurfaces1986} and fully established in \cite{bryantConformalVolume2tori2015}, it is an equality in the restricted domain satisfying
	$$
	 (x-\frac{1}{2})^2+y^2\leq \frac{9}{4}.
	$$
In this region, the minimum is achieved by the rectangular torus with $\omega=i \sqrt{2} $ and the maximum is achieved by the hexagonal torus with $\omega=\frac{1}{2} +i \frac{\sqrt{3}}{2} $ and the values satisfy
	$$
\lambda_{CM}^2[\mathbb{S}^2]\approx 1.47<1.48\approx\frac{\sqrt{2}\pi}{3} \leq \frac{\pi y}{x^2+y^2-x+1}\leq \frac{\sqrt{3} \pi}{3} \approx 1.814.
$$
\end{rem}

\subsection{Virtual entropy}
In order to establish the link between the Li-Yau conformal volume and the Colding-Minicozzi entropy, it will be helpful to introduce an alternative notion of entropy.  This is inspired by work of Magni-Mantegazza \cite{magman08},  who, prior to Colding-Minicozzi's paper on the subject, introduce a quantity based off of Hamilton's matrix Harnack inequality \cite{hamiltonMatricHarnackEstimate1993} and its application to a generalized notion of Huisken monotonicity \cite{hamiltonMonotonicityFormulasParabolic1993}.  We revisit the approach of \cite{magman08}.  Using a simple observation about the matrix Harnack inequality in Euclidean space, we are led to a more flexible concept which should have other applications beyond the proof of Theorem \ref{LYCMIneqThm} .  

Let us first briefly recall some observations and ideas drawn from \cite{hamiltonMonotonicityFormulasParabolic1993,magman08}.  For $t>0$ and $u_0\in C^0\cap L^1(\Real^N)$, consider the convolution by the heat kernel,
\begin{equation}\label{HeatFlowEqn}
\begin{aligned}
(\mathcal{H}_t u_0)(\mathbf{x})=\int_{\Real^N} u_0(\mathbf{y})\mathcal{H}(t, \mathbf{x}, \mathbf{y})d\mathcal{L}^N(\mathbf{y});\\
\mathcal{H}(t, \mathbf{x}, \mathbf{y})=(4\pi t)^{-\frac{N}{2}} e^{-\frac{|\mathbf{x}-\mathbf{y}|^2}{4t}}.
\end{aligned}
\end{equation}
 For such $u_0$, $U(t, \cdot)=\mathcal{H}_t u_0$ is a solution of the heat equation continuous on $[0, \infty)\times \Real^N$.   For $t>0$, consider the collection time $t$ evolutions of positive mass one initial data, 
$$
\mathcal{H}^+_1(t)=\set{\mathcal{H}_t u_0, u_0\in C^0(\Real^N),  u_0>0, \int_{\Real^N} u_0 d\mathcal{L}^N=1}\subset  C^\infty\cap L^1(\Real^N),
$$
For any $\delta>0$, the Gaussian profiles satisfy
$$
\mathcal{G}_1(t)= \set{ \mathcal{H}(t, \mathbf{x}, \mathbf{x}_0)=(4\pi t)^{-\frac{N}{2}} e^{-\frac{|\mathbf{x}-\mathbf{x}_0|^2}{4t}}:\mathbf{x}_0\in \Real^N}\subset \mathcal{H}^+_1(t+\delta).
$$
For $t>0$, any $U\in \mathcal{H}_1^+(t)\cup \mathcal{G}_1(t)$, satisfies Hamilton's matrix Harnack inequality:
\begin{equation}\label{LogConvexEqn}
	\nabla^2_{\Real} \log U +\frac{1}{2t} g_{\Real}\geq 0.
\end{equation}
When traced this is the Li-Yau Harnack inequality.  Both inequalities hold  in more general settings -- see \cite{hamiltonMatricHarnackEstimate1993, liParabolicKernelSchr6dinger1986}.  
For the convenience of the reader, an elementary proof of this fact in Euclidean space and for positive solutions given by convolving with the heat kernel as in \eqref{HeatFlowEqn} is provided in Lemma \ref{MHInequProp}.

 Hamilton observed in \cite{hamiltonMonotonicityFormulasParabolic1993}, that \eqref{LogConvexEqn} could allowed one to generalize the Huisken monotonicity formula \cite{HuiskenMon}  -- see also \cite{Ecker2001}.  It was this observation that was used in \cite{magman08} and which we will also exploit.  Following Magni-Mantegazza, for $t>0$ and $\Sigma\subset \Real^N$ an $n$-dimensional submanifold set
$$
\sigma_{MM}[\Sigma, t]= \sup_{U\in \mathcal{H}^+_1(t)} (4\pi t)^{\frac{N-n}{2}} \int_{\Sigma} U d\mathcal{H}^n.
$$
It is shown in \cite{magman08} that it suffices to take the sup over $U\in \mathcal{G}_1(t)$ and so
$$
\sigma_{MM}[\Sigma, t]=\sup_{U\in \mathcal{G}_1(t)}(4\pi t)^{\frac{N-n}{2}} \int_{\Sigma} U d\mathcal{H}^n\leq\lambda_{CM}^n[\Sigma]=\sup_{t>0} \sigma_{MM}[\Sigma, t].
$$
That is,  the Magni-Mantegezza functionals lead to the Colding-Minicozzi entropy.

We modify the approach of Magni-Mantegazza by focusing on the condition \eqref{LogConvexEqn}.  To that end, for a positive function $U\in C^2(\Real^{N})$, set
$$
\tau(U)= \sup\set{\tau\geq 0:  2\tau \nabla^2_\Real \log U+  g_\Real \geq 0}
$$
which we call the \emph{virtual time} of $U$. Note that $\tau(U)\geq 0$, but can be equal to $0$.  When $U\in \mathcal{H}^+_1(t)\cup \mathcal{G}_1(t)$, $t>0$, \eqref{LogConvexEqn} holds and so
$$
\tau(U) \geq t>0
$$
and one has equality for $U\in \mathcal{G}_1(t)$.  That is, for positive solutions of the heat equation the virtual time estimates the real time from above.  When $u_0\in C^2\cap L^1(\Real^N)$ has $\tau(u_0)=\tau_0>0$ and $U(t,\cdot)=\mathcal{H}_t u_0$ is heat flow with initial data $u_0$ the proof of matrix Harnack inequality from \cite{hamiltonMatricHarnackEstimate1993} suggests that
\begin{equation}\label{LinGrowVTEqn}
\tau(U(t,\cdot))\geq t+\tau_0.
\end{equation}
That is, the virtual time of a positive solution of the heat equation grows at least linearly in the evolution time.
This is slightly subtle due to the non-compactness of $\Real^N$ and is proved in Proposition \ref{MHInequProp} using ideas from \cite{caoMatrixLiYauHamiltonEstimates2005}.

For $T>0$, set
$$
\mathcal{VT}_1^+(T) =\set{u \in  C^2\cap L^1(\Real^{N}): u>0,  \tau (u)\geq T, \int_{\Real^N} u d\mathcal{L}^N=1}.
$$
Clearly, for $0<T_1<T_2$, $\mathcal{VT}_1^+(T_2)\subset \mathcal{VT}_1^+(T_1)$. As observed above
\begin{equation}\label{EinHinVEqn}
\mathcal{H}_1^+(T)\cup \mathcal{G}_1(T)\subset \mathcal{VT}_1^+(T).
\end{equation}\
The inclusion is strict as the elements of $\mathcal{H}_1^+(T)$ must also have strong regularity properties not required of elements in $\mathcal{VT}_1^+(t)$.   Moreover, $\mathcal{G}_{1}(t)\subset \mathcal{VT}_1^+(T)$ for $t\geq T$ and $\mathcal{G}_1(t)\cap \mathcal{VT}_1^+(T)=\emptyset $ for $0<t<T$
 Standard facts about heat flow and \eqref{LinGrowVTEqn} also imply,
$$
\mathcal{H}_t(\mathcal{VT}_1^+(T))\subset \mathcal{VT}_1^+(T+t).
$$
Given an $n$-dimensional submanifold $\Sigma \subset \Real^N$ we define its \emph{\vte} to be
$$
\lambda_{V}^n[\Sigma]=\sup_{T>0}\sup_{U\in\mathcal{VT}_1^+(T)} \left( (4\pi T )^{\frac{N-n}{2}} \int_\Sigma U d\mathcal{H}^n\right).
$$
As in \cite{hamiltonMonotonicityFormulasParabolic1993, magman08}, this quantity is monotone along reasonable mean curvature flows.
\begin{thm}\label{LCMonThm}
	If $t\in [0, T)\mapsto \Sigma_t\subset \Real^N$ is a mean curvature flow of $n$-dimensional proper submanifolds with $\lambda_{V}^n[\Sigma_0]<\infty$, then 
	$$
	t\mapsto \lambda_{V}^n[\Sigma_t]
	$$
	is monotone non-increasing in $t$.
\end{thm}

It follows immediately from \eqref{EinHinVEqn} that
$$
\lambda_{CM}^n[\Sigma]\leq  \lambda_{V}^n[\Sigma].
$$
The long-time behavior of suitable solutions of the heat equation is towards the Gaussian -- see,  e.g., \cite{vazquezAsymptoticBehaviourMethods2018}.  Using Theorem \ref{LCMonThm}, this suggests that when $\Sigma$ is a self-shrinker,
\begin{equation}
\label{SelfShrinkEntEqualEqn}
\lambda_{CM}^n[\Sigma]=  \lambda_{V}^n[\Sigma].
\end{equation}
This is proved in Theorem \ref{SSVTCMThm}.  It is unclear to the author how general this equality is.

\subsection{Stable conformal volume}
Finally, we introduce a functional obtained by ``stabilizing" the (normalized) Li-Yau conformal volume in a manner reminiscent of the way stable homotopy groups are introduced.  This is distinct from the stabilization procedure of \cite{liNewConformalInvariant1982a} which involved conformally embedding in higher and higher dimensional spheres. This quantity is related to both the Colding-Minicozzi and virtual entropies.

Given a $n$-dimensional submanifold $\Sigma\subset \Real^N$, its \emph{stable conformal volume} is defined as
$$
\bar{\lambda}^\infty_{LY} [\Sigma]=\liminf_{m\to \infty} \bar{\lambda}_{LY}^{n+2m}[\Sigma\times \Real^{2m}].
$$
Crucially, this sequence is monotonically increasing in the stabilization parameter:
\begin{prop}\label{StabConfVolMonDimProp}
	If $\Sigma \subset \Real^N$ is a $n$-dimensional proper submanifold, then
	$$
	\bar{\lambda}_{LY}^n[\Sigma]\leq \bar{\lambda}_{LY}^{n+2m}[\Sigma \times \Real^{2m}]\leq \bar{\lambda}^\infty_{LY} [\Sigma].
	$$
	In particular, $m\mapsto \bar{\lambda}_{LY}^{n+2m}[\Sigma \times \Real^{2m}]$ is increasing in $m$ and so
	$$
	\bar{\lambda}^\infty_{LY}[\Sigma]= \lim_{m \to \infty} \bar{\lambda}_{LY}^{n+2m}[\Sigma\times \Real^{2m}].
	$$
\end{prop}
In contrast with process in \cite{liNewConformalInvariant1982a}, this method of stabilization breaks the full conformal invariance -- though the functionals remain invariant under Euclidean translations and dilations. For instance, it follows from \cite{bryantSurfacesConformalGeometry1988} that
$$
\bar{\lambda}_{LY}^{n+2}[\Sigma\times \Real^2] =1
$$
if and only if $\Sigma$ is an affine $n$-plane.  In particular, even though $\bar{\lambda}_{LY}^n[\Real^n]=\bar{\lambda}_{LY}^n[\mathbb{S}^n]=1$, 
$$
1=\bar{\lambda}_{LY}^{n}[\mathbb{S}^n]<\bar{\lambda}_{LY}^{n+2}[\mathbb{S}^n\times \Real^2].
$$

Rather remarkably, the stable conformal volume lies between the Colding-Minicozzi entropy and the \vte.
\begin{thm}\label{CMLYVTIneqThm}
	For any $n$-dimensional proper submanifold, $\Sigma \subset \Real^N$, 
	$$
	\lambda_{CM}^n[\Sigma]\leq \bar{\lambda}_{LY}^\infty[\Sigma]\leq \lambda_{V}^n[\Sigma].
	$$
	That is, the Colding-Minicozzi entropy of $\Sigma$ provides a lower bound for its stable conformal volume while the virtual entropy  provides an upper bound.
\end{thm}
These inequalities are sharp as, when $\Sigma$ is a self-shrinker, appealing to \eqref{SelfShrinkEntEqualEqn} yields
$$
	\lambda_{CM}^n[\Sigma]= \bar{\lambda}_{LY}^\infty[\Sigma]= \lambda_{V}^n[\Sigma].
$$
Theorem \ref{LYCMIneqThm} follows immediately from this, Proposition \ref{StabConfVolMonDimProp} and Theorem \ref{CMLYVTIneqThm}.

\section{Properties of elements of $\mathcal{VT}_1^+(T)$ and their heat flows}
In \cite{hamiltonMatricHarnackEstimate1993}, Hamilton used a (tensor) maximum principle to prove the matrix Harnack inequality \eqref{LogConvexEqn} for positive solutions of the heat equation.  We use the same idea to establish a linear in (real) time growth of the measure of convexity, i.e., the $t$ in \eqref{LogConvexEqn}.  In \cite{hamiltonMatricHarnackEstimate1993} the underlying manifold is closed and so, as we work in Euclidean space, we instead adapt an argument of \cite{caoMatrixLiYauHamiltonEstimates2005}.  While \cite{hamiltonMatricHarnackEstimate1993} treats quite general positive solutions of the heat equation, in order to simplify exposition we focus on the subclass of solutions given by \eqref{HeatFlowEqn} with initial data in $\mathcal{VT}_1^+(T)$, $T>0$. As such, we establish some basic properties of elements of $\mathcal{VT}_1^+(T)$ and of their evolution by \eqref{HeatFlowEqn}.

We first give an elementary proof that \eqref{LogConvexEqn} holds for elements of $\mathcal{H}_1^+(t)$.
\begin{lem}\label{LogConvexLem}
	For $u_0\in C^0\cap L^1(\Real^N)$ with $u_0\geq 0$ and $\int_{\Real^N} u_0 d\mathcal{L}^N=1$, the function
$
	U(t, \cdot )=\mathcal{H}_t u_0\in \mathcal{H}^+_1(t)$ given by \eqref{HeatFlowEqn}
	satisfies, for $t>0$,
	$$
\nabla_\Real^2 \log U(t,\mathbf{x})\geq  -\frac{1}{2t} g_{\Real}.	
	$$
%Equality holds identically if and only if $U(t, \mathbf{x})=\mathcal{H}(t, \mathbf{x}, \mathbf{x}_0)$ for some $\mathbf{x}_0\in \Real^N$.
\end{lem}
\begin{proof}
	The hypotheses on $u_0$ ensure that, for $t>0$, 
	$$
	U>0 \mbox{ and }\int_{\Real^N} U(t, \mathbf{x})d\mathcal{L}^N=1.
	$$   
	As $u_0$ is integrable and, for $t>0$, $|\nabla^i_{\Real, \mathbf{x}} \mathcal{H}(t, \mathbf{x}, \mathbf{y})|\leq C_{N, i} t^{-\frac{N+2i}{2}}$,  one may differentiate under the integral sign and obtain for any parallel unit vector field $\mathbf{v}$ that, 
	\begin{align*}
		\nabla_\Real^2 &U(t, \mathbf{x})(\mathbf{v}, \mathbf{v}) = \int_{\Real^N} u_0(\mathbf{y}) \nabla^2_{\Real, \mathbf{x}} \mathcal{H}(t, \mathbf{x}, \mathbf{y}) (\mathbf{v}, \mathbf{v}) d\mathcal{L}^N(\mathbf{y})\\
		&=\int_{\Real^N} u_0(\mathbf{y}) \left( \frac{(\mathbf{v} \cdot\nabla_{\Real, \mathbf{x}} \mathcal{H}(t, \mathbf{x}, \mathbf{y}))^2} {\mathcal{H}(t,\mathbf{x}, \mathbf{y})^2}-\frac{|\mathbf{v}|^2}{2t} \right) \mathcal{H}(t, \mathbf{x}, \mathbf{y}) d\mathcal{L}^N(\mathbf{y})\\
		&= -\frac{U(t, \mathbf{x}) }{2t} |\mathbf{v}|^2 +\int_{\Real^N}  \left( u_0^{\frac{1}{2}}(\mathbf{y})  \frac{\mathbf{v} \cdot\nabla_{\Real, \mathbf{x}} \mathcal{H}(t, \mathbf{x}, \mathbf{y})} {\mathcal{H}(t,\mathbf{x}, \mathbf{y})} \right)^2 \mathcal{H}(t,\mathbf{x}, \mathbf{y}) d\mathcal{L}^N(\mathbf{y})  \\
		&\geq -\frac{U(t, \mathbf{x}) }{2t} |\mathbf{v}|^2 +U^{-1}(t, \mathbf{x})\left(\int_{\Real^N} u_0(\mathbf{y})  \frac{\mathbf{v} \cdot\nabla_{\Real, \mathbf{x}} \mathcal{H}(t, \mathbf{x}, \mathbf{y})} {\mathcal{H}(t,\mathbf{x}, \mathbf{y})}\mathcal{H}(t,\mathbf{x}, \mathbf{y})d\mathcal{L}^N(\mathbf{y})\right)^2\\
		&=-\frac{U(t, \mathbf{x}) }{2t} |\mathbf{v}|^2+U^{-1}(t, \mathbf{x})(\mathbf{v}\cdot \nabla_\Real U(t, \mathbf{x}))^2.
	\end{align*}
That is, for $t>0$, we get the claimed inequality
$$
\nabla^2_\Real \log U(t, \mathbf{x}) \geq - \frac{1}{2t} g_{\Real},
$$
\end{proof}

We next establish a useful pointwise bound on elements of $\mathcal{VT}_1^+(T)$.
\begin{lem}\label{L12PWEstLem}
There is a $C_0=C_0(N)>0$ so that, for $f\in \mathcal{VT}_1^+(T)$ and $\mathbf{x}_0\in \Real^N$, 
$$
 f(\mathbf{x}_0)\leq  C_0 T^{-\frac{N}{2}} \int_{B_{\sqrt{T}}(\mathbf{x}_0)} f d\mathcal{L}^N\leq   C_0 T^{-\frac{N}{2}}.
 $$
\end{lem}
\begin{proof}
	The fact that $f\in \mathcal{VT}_1^+(T)$ implies
	$$
	-\frac{N}{2T}\leq \Delta_{\Real} \log f .
	$$
	Hence, for the given $\mathbf{x}_0$,  the function
	$$
	F_{\mathbf{x}_0}(\mathbf{x})= \log f(\mathbf{x})+\frac{|\mathbf{x}-\mathbf{x}_0|^2}{4T}
	$$
	is sub harmonic and so, by the mean value inequality, for any $R>0$,
	$$
	\log f (\mathbf{x}_0)=\log F_{\mathbf{x}_0}(\mathbf{x}_0)\leq \frac{1}{\omega_N R^N} \int_{B_R(\mathbf{x}_0)} \log f + \frac{|\mathbf{x}-\mathbf{x}_0|^2}{4T}d\mathcal{L}^N,
	$$
	where $\omega_N=|B_1|$, the volume of unit $N$-ball.
	Taking $R=\sqrt{T}$, yields
	$$ 
	\log f(\mathbf{x}_0)\leq \frac{1}{\omega_N T^{\frac{N}{2}}} \int_{B_{\sqrt{T}}(\mathbf{x}_0)} \log f d\mathcal{L}^n + \frac{1}{4 (N+2)}.
	$$
	It follows from Jensen's inequality that
	$$
	f(\mathbf{x}_0)\leq C_0 T^{-\frac{N}{2}}  \int_{B_{\sqrt{T}}(\mathbf{x}_0)} f d\mathcal{L}^N, \mbox{ where } C_0=C_0(N)= \omega_N^{-1}\exp( \frac{1}{4 (N+2)}).
	$$
  To complete the proof, observe that as $f\geq 0$, 
	$$
	f(\mathbf{x}_0)\leq  \int_{B_{\sqrt{T}}(\mathbf{x}_0)} f d\mathcal{L}^N\leq  C_0 T^{-\frac{N}{2}} \int_{\Real^N} f d\mathcal{L}^N=   C_0 T^{-\frac{N}{2}}.
	$$
\end{proof}
We now establish some useful estimates for elements of $\mathcal{VT}_1^+(T)$.
\begin{lem}\label{BoundLem}
	Suppose $u\in \mathcal{VT}_1^+(T)$, $T>0$.  The following hold:
	\begin{enumerate}
		\item  $\lim_{|\mathbf{x}|\to \infty} u(\mathbf{x})=0$ and there is a point $\mathbf{z}_0=\mathbf{z}_0(u)\in \Real^N$ so that
		$$0< u(\mathbf{z}_0) e^{-\frac{|\mathbf{x}-\mathbf{z}_0|^2}{4T}}\leq u(\mathbf{x})\leq u(\mathbf{z}_0) \leq (4\pi T)^{-\frac{N}{2}};$$
		\item $ \lim_{|\mathbf{x}|\to \infty} |\nabla_\Real u(\mathbf{x})|=0$ and	
		$$|\nabla_\Real u(\mathbf{x})|^2\leq \frac{4\pi}{e} (4\pi T)^{-\frac{N+2}{2}} u(\mathbf{x}) \leq \frac{4\pi}{e}  (4\pi T)^{-N-1}; $$
		\item $|\nabla^2_\Real u|\in L^1(\Real^N)$ and there is a constant $C_1=C_1(N)>0$ so
		$$
		\int_{\Real^N} |\nabla^2_\Real u|d \mathcal{L}^N\leq C_1 T^{-1}.
		$$
	\end{enumerate}
\end{lem}
\begin{rem}
		In particular, for  $u\in \mathcal{VT}_1^+(T)$, $T>0$,  $\tau(u)<\frac{1}{4\pi}u(\mathbf{z}_0)^{-\frac{2}{N}}<\infty$.
\end{rem}
\begin{proof}
	As $u\in L^1(\Real^N)$ and $u>0$, for every $\epsilon>0$, there is $R_{\epsilon}$ so that
	$$
 0<	\int_{\Real^N\setminus B_{R_\epsilon}} u d\mathcal{L}^N <\epsilon.
	$$
It follows from Lemma \ref{L12PWEstLem} that, for any $\mathbf{x}_0\not\in B_{R_\epsilon+\sqrt{T}}$,
	$$
	u(\mathbf{x}_0)\leq  C_0 T^{-\frac{N}{2}}\int_{\Real^N\setminus B_{R_\epsilon}} u d\mathcal{L}^N \leq C_0 T^{-\frac{N}{2}} \epsilon.
	$$
In other words,
	$$
	\lim_{|\mathbf{x}| \to \infty} u(\mathbf{x})=0.
	$$
As $u$ is also continuous, there is a $\mathbf{z}_0$ so that
	$$
	\max_{\mathbf{x}\in \Real^N} u(\mathbf{x})=u(\mathbf{z}_0).
	$$
		Hence, $\nabla_\Real u(\mathbf{z}_0)=\mathbf{0}$.
	Observe that as $U\in \mathcal{VT}_1^+(T)$, $u$ satisfies the convexity property given by \eqref{LogConvexEqn} with $t=T$. This implies that
	$$
	-\frac{|\mathbf{x}-\mathbf{z}_0|^2}{4T} +\log u(\mathbf{z}_0) \leq \log u(\mathbf{x}).
	$$
	This yields the claimed lower bound $u$. The upper bound on $u(\mathbf{z}_0)$, and hence on $u$, follows by integrating  over $\Real^N$ and using that $u$ has mass one. We have established Item (1).

	For the gradient estimate pick $\mathbf{x}_0\in \Real^N$. 
	Set $\nabla_\Real \log  u(\mathbf{x}_0)=\mathbf{v}$ and 
	$$
	Q(\mathbf{x})=-\frac{1}{4T} |\mathbf{x}-\mathbf{x}_0-2T \mathbf{v}|^2 +\log u(\mathbf{x}_0)+T|\mathbf{v}|^2.
	$$
	Thus,  $Q(\mathbf{x}_0)=\log u(\mathbf{x}_0),$ $\nabla_\Real Q(\mathbf{x}_0)=\mathbf{v}=\nabla_\Real \log u(\mathbf{x}_0)$ and $\nabla^2_\Real Q=-\frac{g_\Real}{2T}$.
	The convexity  of $\log u$, i.e., \eqref{LogConvexEqn}, ensures that
	$$
	Q(\mathbf{x})\leq \log u(\mathbf{x}).
	$$
%	Hence,
%	$$
%	\log u(\mathbf{x}_0)+T|\mathbf{v}|^2\leq Q(\mathbf{x}_0+2T\mathbf{v})\leq \log u(\mathbf{x}_0+2T\mathbf{v})\leq -\frac{N}{2} \log 4\pi T.
%	$$
% While we don't use it,  we note the lower bound on $u$ yields
%	$$
%|\mathbf{v}|^2 \leq \frac{|\mathbf{x}_0-\mathbf{z}_0|^2}{4T^2} \mbox{ i.e., }  |\nabla_\Real u(\mathbf{x}_0)|\leq 	\frac{|\mathbf{x}_0-\mathbf{z}_0|}{2T} u(\mathbf{x}_0). 
%	$$
Hence, integrating $\exp(Q)$ and using the mass of $u$ is one yields, 
	$$
	(4\pi T)^{\frac{N}{2}}u(\mathbf{x}_0) e^{T|\mathbf{v}|^2} \leq 1.
	$$
Using the elementary identity
	$$
	-x \log x \leq e^{-1}
	$$
	and the upper bound on $u$ yields the claimed bounds
	\begin{align*}
		|\nabla u(\mathbf{x}_0)|^2&\leq - T^{-1} u^2(\mathbf{x}_0)\log \left((4\pi T)^{\frac{N}{2}} u(\mathbf{x}_0)\right)\\
		&\leq \frac{4\pi}{e} (4\pi T)^{-\frac{N}{2}-1}u(\mathbf{x}_0)\leq \frac{4\pi}{e} (4\pi T)^{-N-1}.
	\end{align*}
Moreover, as $u(\mathbf{x})$ tends to zero as $|\mathbf{x}|\to \infty$ this implies same is true of $\nabla u(\mathbf{x})$. This completes proof of Item (2).
While we don't use it,  we note the lower bound on $u$ yields
	$$
|\nabla_\Real u(\mathbf{x}_0)|^2\leq 	\frac{|\mathbf{x}_0-\mathbf{z}_0|^2}{4T^2} u^2(\mathbf{x}_0). 
	$$	
	
	%Alternatively,
	%\begin{align*}
	%|\nabla u(\mathbf{x}_0)|^2&\leq \left( \log C_N+ \frac{N}{2}\log T \right) T^{-1} u(\mathbf{x}_0)^2 - \frac{1}{T} u(\mathbf{x}_0)^2 \log u(\mathbf{x}_0)\\
	%&\leq \left( \log C_N+ \frac{N}{2}\log T \right)T^{-1} u(\mathbf{x}_0)^2+ T^{-1} e u(\mathbf{x}_0) 
	%\end{align*}
	%When $T=1$ this gives a uniform bound and this is enough as we can rescale $u$ so
	%$$
	%T^{\frac{N}{2}}u(\sqrt{T} \mathbf{x})\in \mathcal{LC}(1).
	%$$
	
	For Item (3), we first observe that, in the sense of symmetric matrices, 
	$$
	-\frac{1}{2T} g_\Real \leq \nabla^2_\Real \log u = \frac{\nabla^2_\Real u}{ u} -\frac{\nabla_\Real u \otimes \nabla_\Real u}{u^2}\leq \frac{\nabla^2_\Real u}{ u}. 
	$$
By Item (1), $0<u<(4\pi T)^{-\frac{N}{2}}$ so
\begin{equation} \label{LinLowBndEqn}
	-  2\pi (4\pi T)^{-\frac{N+2}{2}} g_\Real\leq -\frac{u}{2T} g_\Real \leq \nabla^2_\Real u \mbox{ and } -\frac{N}{2T } u \leq \Delta_\Real u.
\end{equation}
An immediate consequence of this is that
$$
 |\Delta_\Real u| \leq \Delta_\Real u +N T^{-1}u.
	$$
For $R>0$, let $\phi_R=(1-(4R)^{-2}|\mathbf{x}|^2)^2$ this is a smooth function which vanishes to first order on $\partial B_{2R}$ and satisfies $\phi_R\leq 1$ and $\Delta_\Real \phi_R \leq C(N)R^{-2}$ on $B_{2R}$ where $C(N)>0$ is a dimensional constant. 
The above inequalities and an integration by parts implies
\begin{align*}
 \frac{1}{4} &\int_{B_R} |\Delta_\Real u| d\mathcal{L}^N \leq \int_{B_{2R}} \phi_R |\Delta_\Real u| d\mathcal{L}^N \leq \int_{B_{2R}} \phi_R \left( \Delta_\Real u +N T^{-1} u \right) d\mathcal{L}^N\\
 &= \int_{B_{2R}} \left(\Delta_\Real \phi_R +N T^{-1}\phi_R\right) u d\mathcal{L}^N\leq \left(C(N) R^{-2}+N T^{-1}\right) \int_{B_{2R}}  u d\mathcal{L}^N.
\end{align*}
As $u$ has mass one, sending $R\to \infty$ shows that $\Delta_\Real u\in L^1(\Real^N)$ and 
$$
\int_{\Real^N} |\Delta_\Real u| d\mathcal{L}^N \leq 4 NT^{-1}.
$$
	Finally, we observe that the lower bounds on the eigenvalues  of $\nabla^2_\Real u$ imply that
	$$
	|\nabla^2_\Real u|^2\leq N\left(|\Delta_\Real u|+\frac{(N-1) u}{2T}\right)^2 \mbox{ and so } |\nabla^2_\Real u|\leq \sqrt{N} |\Delta_\Real u|+ \frac{N^{\frac{3}{2}}}{2T} u.
	$$
The final claim follows after making an appropriate choice of $C_1=C_1(N)>0$.
\end{proof}
Next we establish some properties of the evolution by \eqref{HeatFlowEqn} of elements in $\mathcal{VT}_1^+(\tau_0)$.
\begin{lem}\label{HeatFlowBoundLem}
	For $\tau_0>0$ and $u_0\in \mathcal{VT}_1^+(\tau_0)$, let $U(t, \cdot )=\mathcal{H}_t u_0$
	be the solution of the heat equation on $[0, \infty)\times \Real^N$ given by \eqref{HeatFlowEqn}.   This solution has the following properties:
	\begin{enumerate}
		\item $\int_{\Real^N} U(t, \mathbf{x})=1$, for all $t\geq 0$;
		\item $U(t,\mathbf{x})\geq u_0(\mathbf{z}_0) (4\pi (t+\tau_0))^{-\frac{N}{2}} e^{-\frac{|\mathbf{x}-\mathbf{z}_0|^2}{4(t+\tau_0)}}>0.$
			\item $U(t,\mathbf{x})+\tau_0^{\frac{1}{2}} |\nabla_\Real U(t, \mathbf{x})|\leq C_2 \tau_0^{-\frac{N}{2}}$, for all $t\geq 0$; 
			\item $|\nabla^2_\Real U(t, \mathbf{x})|\leq C_2 \tau_0^{-1} t^{-\frac{N}{2}}$ for all $t>0$;
		\end{enumerate}
	Here $C_2=C_2(N)>0$ and $\mathbf{z}_0=\mathbf{z}_0(u_0)$ is given by Lemma \ref{BoundLem}.
	Moreover,
	\begin{enumerate}[resume]
		\item $\nabla^2_\Real U(t, \mathbf{x})\geq -\frac{U(t,\mathbf{x})}{2\tau_0} g_\Real$ for all $t\geq 0$; 
			\item $\lim_{t\to 0^+} U(t, \cdot) =u_0$, in $L^1(\Real^N)$, $C^1(\Real^N)$ and $C^2_{loc}(\Real^N)$;
		\item $U, \nabla_\Real U, \nabla^2_\Real U\in C^0([0, \infty)\times \Real^N).$
	\end{enumerate}
\end{lem}
\begin{rem}
	In fact, $U$ is the unique continuous and non-negative solution of the heat equation on $[0, \infty)\times \Real^N$ with initial data $u_0$ --  see \cite{Dodziuk, donnellyUniquenessPositiveSolutions1987, liParabolicKernelSchr6dinger1986}.
\end{rem}
\begin{proof}
	The first item is a consequence of the fact that $u_0$ has mass one and Fubini's theorem.    The second follows from Item (1) of  Lemma \ref{BoundLem} and an explicit integration of \eqref{HeatFlowEqn}.
	
	For the remaining items we observe that as $u_0$ is integrable and for $t>0$ and $i\geq 0$,
	$$
	|\nabla_{\Real, \mathbf{x}}^i\mathcal{H}(t, \mathbf{x}, \mathbf{y})|\leq C_{N, i} t^{-\frac{N+2i}{2}} 
	$$
	we may differentiate under the integral sign and obtain
	$$
	\nabla^i_{\Real} U(t ,\mathbf{x})= \int_{\Real^N} u_0(\mathbf{y}) \nabla^{i}_{\Real, \mathbf{x}} \mathcal{H}(t, \mathbf{x}, \mathbf{y}) d\mathcal{L}^N.
	$$
	As the symmetries of the heat kernel ensure
	$$
	\nabla_{\Real, \mathbf{x}}^i\mathcal{H}(t, \mathbf{x}, \mathbf{y}) = (-1)^i \nabla_{\Real, \mathbf{y}}^i\mathcal{H}(t, \mathbf{x}, \mathbf{y}),
	$$
	The uniform boundedness of $u_0$ and $\nabla_\Real u_0$ and fact that $|\nabla^2_{\Real} u_0|\in L^1(\Real^N)$ allows one to integrate by parts and obtain, for a fixed parallel unit vector field $\mathbf{v}$ and $t>0$,
	$$
	\mathbf{v}\cdot \nabla_\Real U(t, \mathbf{x})= \int_{\Real^N} \mathbf{v}\cdot \nabla_\Real u_0 (\mathbf{y}) \mathcal{H}(t, \mathbf{x}, \mathbf{y})d\mathcal{L}^N(\mathbf{y});
	$$
	$$
	\nabla^2_\Real U(t, \mathbf{x})(\mathbf{v}, \mathbf{v})= \int_{\Real^N} \nabla^2_\Real u_0(\mathbf{v}, \mathbf{v}) \mathcal{H}(t, \mathbf{x}, \mathbf{y}) d\mathcal{L}^N (\mathbf{y}).
	$$
	Items (3) is then immediate consequences of the uniform bounds of Item (1) and (2) of Lemma \ref{BoundLem}.  Likewise, Item (4) follows directly from the integral bound of Item (3) of Lemma \ref{BoundLem}.   It is also clear Item (5) follows from \eqref{LinLowBndEqn} and representation formula \eqref{HeatFlowEqn}.
	
As both $u_0$ and $\nabla_\Real u_0$ are continuous and tend to zero at infinity, they are both uniformly continuous and bounded.  Standard results about approximations to the identity give
  $$
  \lim_{t\to 0^+} \Vert U(t, \cdot) -u_0(\cdot)\Vert_{\infty}+ \Vert \nabla_{\Real} U(t, \cdot) -\nabla_\Real u_0 (\cdot) \Vert_\infty =0.
  $$
  It is also an elementary fact that because $|\nabla^2_\Real u_0|\in L^1(\Real^N)$ and $u_0$ is $C^2$, for any $\mathbf{x}_0$
  $$
  \lim_{t\to 0^+ ,\mathbf{x}\to \mathbf{x}_0} \nabla^2_\Real U(t, \mathbf{x})=\nabla^2_\Real u_0(\mathbf{x}_0).
  $$
Moreover, this convergence can be taken to be uniform on any compact subset of $\Real^N$.  Together this gives Item (6).  Item (7) is an immediate consequence.
\end{proof}

We can now establish the linear increase in virtual time under heat flow.
\begin{prop}\label{MHInequProp}
	Suppose $u_0\in \mathcal{VT}_1^+(\tau_0)$ for $\tau_0>0$ and let $U(t, \cdot )=\mathcal{H}_t u_0$
	be the solution of the heat equation on $[0, \infty)\times \Real^N$ given by \eqref{HeatFlowEqn}.  
	For all $t\geq 0$, 
	$$
	U(t, \cdot)\in \mathcal{VT}_1^+(\tau_0+t).
	$$
 Moreover, if $2(\tau_0+t_0) \nabla^2_\Real \log U (t_0, \mathbf{x}_0)+g_\Real =0$ for some $t_0>0$ and $\mathbf{x}_0 \in \Real^N$, then $U(t, \cdot)=\mathcal{H}(\tau_0+t, \cdot, \mathbf{x}_0')\in  \mathcal{G}_1(\tau_0+t)$ for some $\mathbf{x}_0'\in \Real^N$ and all $t\geq 0$.
\end{prop}
%\begin{rem}
%	This should hold more generally for with $\Real^N$ replaced by $(M,g)$ where $M$ has non-negative sectional curvatures and parallel Ricci curvature. 
%\end{rem}
\begin{proof}
	The positivity and unit mass condition are preserved by Lemma \ref{HeatFlowBoundLem}.  Hence, it suffices to establish that, for all $t>0$,
	$$
	(\tau_0+t)\nabla^2_\Real \log U(t, \mathbf{x})+g_\Real \geq 0.
	$$
	To that end, fix a parallel unit length vector field $\mathbf{v}$ on $\Real^{N}$.  By the rotational symmetry of the problem, is enough to take $\mathbf{v}=\mathbf{e}_1$.  Set
	\begin{align*}
		H(t,\mathbf{x})&=U(t,\mathbf{x})\nabla^2_{\Real} \log U(t,\mathbf{x})(\mathbf{v}, \mathbf{v})= U(t,\mathbf{x})\partial^2_1 \log U(t,\mathbf{x})\\
		&= \partial^2_1 U(t,\mathbf{x}) - \frac{(\partial_1 U(t,\mathbf{x}))^2}{U(t,\mathbf{x})}.
	\end{align*}
	The hypotheses ensure that
	$$
2 \tau_0	H(0,\mathbf{x}) +u_0(\mathbf{x})\geq 0.
	$$
	We compute that
	\begin{align*}
		&\partial_t H= \partial_1^2 \partial_t U -\frac{2 \partial_1 U \partial_1 \partial_t U}{U}+ \frac{(\partial_1 U)^2 \partial_t U}{U^2}\\
		&=\partial_1^2\Delta_\Real  U -\frac{2 \partial_1 U \partial_1 \Delta_\Real U}{U}+ \frac{(\partial_1 U)^2 \Delta_\Real U}{U^2}\\
		&= \Delta_\Real \partial_1^2 U-\frac{\Delta_\Real (\partial_1 U)^2-2|\nabla_\Real \partial_1 U|^2}{U}- (\partial_1 U)^2 (\Delta_\Real U^{-1}) +2\frac{ (\partial_1 U|\nabla_\Real U|)^2}{U^3}\\
		&= \Delta_\Real \partial_1^2 U-\Delta_\Real \left(\frac{(\partial_1 U)^2}{U}\right)+\frac{2|\nabla_\Real \partial_1 U|^2}{U}-\frac{4 \partial_1 U \nabla_\Real U\cdot \nabla_\Real \partial_1 U}{U^2} +2  \frac{(\partial_1 U|\nabla_\Real U|)^2}{U^3}\\
		&= \Delta_\Real H+2 U^{-1}|\nabla_\Real \partial_1 U-U^{-1} \partial_1 U\nabla_\Real U|^2\geq \Delta_\Real H+2 U^{-1}\left| \partial_1^2 U-U^{-1} (\partial_1 U)^2\right|^2\\
		&= \Delta_\Real H+2U^{-1} H^2.
	\end{align*}
%	Where we used
%	$$
%	\nabla_\Real \partial_1 U = \partial_1^2 U \frac{\partial}{\partial x_1} + \mathbf{w}
%	$$
%	for $\mathbf{w}$ is orthogonal to $\mathbf{v}= \frac{\partial}{\partial x_1} $.
	
	Now set
	$$
	N(t, \mathbf{x})= 2(t+\tau_0) H(t, \mathbf{x})+ U(t,\mathbf{x})
	$$
	we see that $N(0, \mathbf{x})\geq 0$ and
	$$
	(\partial_t -\Delta_{\Real}) N\geq  2 H+4(t+\tau_0) U^{-1} H^2=2U^{-1} H N
	$$
	
	As $U(0, \mathbf{x})=u_0\in \mathcal{VT}_1^+(\tau_0)$, Lemmas \ref{BoundLem} and \ref{HeatFlowBoundLem} together imply that
\begin{align*}
	H(t,\mathbf{x})&\geq -\frac{1}{2\tau_0} u_0(\mathbf{x})-C (t+\tau_0)^{\frac{N}{2}}e^{\frac{|\mathbf{x}-\mathbf{z}_0|^2}{4(t+\tau_0)}}\\
	&\geq -C-C (t+\tau_0)^{\frac{N}{2}}e^{\frac{|\mathbf{x}-\mathbf{z}_0|^2}{4(t+\tau_0)}}
\end{align*}
	where $C=C(\tau_0,N, u_0)>0$ is a constant and $\mathbf{z}_0=\mathbf{z}_0(u_0)\in \Real^N$ is a fixed point.
	
	Now let 
	$$
	\Psi_{T}(t, \mathbf{x})=  e^t (T-t)^{-\frac{N}{2}} e^{\frac{|\mathbf{x}-\mathbf{z}_0|^2}{4(T-t)}}
	$$
	One verifies that $\Psi_T$ satisfies 
	$$
	\left(\partial_t -\Delta_\Real\right) \Psi_T= \Psi_T
	$$
	on $(-\infty, T]$. 
	Moreover, $\Psi_T$ has very rapid growth as $\mathbf{x}\to \infty$.
	
	Set
	$$
	N_{\epsilon}(t, \mathbf{x})= N(t, \mathbf{x})+\epsilon \Psi_{\frac{\tau_0}{2}}(t,\mathbf{x})
	$$
	which is smooth for $t\in [0, \frac{\tau_0}{2})$
	One readily checks that 
	$$
	N_{\epsilon}(0, \mathbf{x}) \geq\epsilon \Psi_{\frac{\tau_0}{2}}(0,\mathbf{x})\geq  T^{-\frac{N}{2}}\epsilon>0  $$
	where $C=C(\tau_0, N)$.
	Likewise, for $t\in ( 0, \frac{\tau_0}{2})$ and outside of a very large ball,
	$$
 N_{\epsilon}(t, \mathbf{x})\geq 	\frac{1}{2}\epsilon \Psi_{\frac{\tau_0}{2}}(t,\mathbf{x})>0 .
	$$
	Now suppose that $N_{\epsilon}(t_0, {\mathbf{x}}_0)<0$ for some $ t_0\in (0, \frac{\tau_0}{2})$.  It follows that there is a value $t_1\in  (0, t_0]$ and $\mathbf{x}_1\in \Real^N$
	$$
0> N_{\epsilon}(t_0, {\mathbf{x}}_0)\geq N_{\epsilon}(t_1, \mathbf{x}_1) = 	\min_{[0, t_1]\times\Real^N} N_\epsilon.
	$$
	Hence, at $(t_1, \mathbf{x}_1)$,
\begin{align*}
 0 &\geq \left(\frac{\partial}{\partial t} -\Delta_\Real\right) N_{\epsilon} =2 H U^{-1} N =2 U^{-1} H N + \epsilon \Psi_{\frac{\tau_0}{2}}>	2H(1+2(t_1+\tau_0) U^{-1 } H),
\end{align*}
which is incompatible with $H(t_1, \mathbf{x}_1)\geq 0$ and so at $(t_1, \mathbf{x}_1)$
	$$
	0>H \mbox{ and } U^{-1} N= 1+2(t_1+\tau_0) U^{-1 } H>0.
	$$
	That is, 
	$$N_\epsilon(t_1, \mathbf{x}_1)>N(t_1, \mathbf{x}_1)>0$$
which is a contradiction.  Hence, $N_\epsilon>0$ on $[0, \frac{\tau_0}{2})$.  By taking $\epsilon \to 0^+$ we see that
	$$
	N(t, \cdot)\geq 0
	$$
	for $t\in (0, \frac{\tau_0}{2})$.  In particular, 
	$$
	U\left(\frac{\tau_0}{4}, \cdot\right) \in \mathcal{VT}_1^+\left(\tau_0+\frac{\tau_0}{4}\right).
	$$
	As $\tau_0+\frac{\tau_0}{4}\geq \tau_0$, one may translate in time and apply the same argument to obtain
	$$
	N(t, \cdot) \geq 0
	$$
	on $[0, \frac{3}{4}\tau_0)$. 
	Iterating this completes the proof that $U(t, \cdot) \in \mathcal{VT}_1^+(\tau_0+t)$.

	Finally, if for some $t_0>0$, $\mathbf{x}_0$,   $2(\tau_0+t_0)\nabla^2_\Real \log U(t_0, \mathbf{x}_0)+g_{\Real}=0$, then for any choice of $\mathbf{v}$ we have that $N\geq 0$ and $N(t_0, \mathbf{x}_0)=0$.  By the strict maximum principle this means $N$ identically vanishes and the claim follows.
	
\end{proof}

While we don't directly use it in this paper,  the fact that the long time behavior of a solution of the heat equation is towards the Gaussian, i.e., $\mathcal{H}(t, \mathbf{x} , \mathbf{0})$ explains some of the results  and so for the reader's benefit we include a precise statement.  This may be thought of as a version of the Central Limit Theorem -- see \cite{vazquezAsymptoticBehaviourMethods2018} for a proof and further discussion.  
\begin{prop}\label{LongTimeHeatProp}
	Suppose $u_0\in \mathcal{VT}_1^+(T)$, $T>0$. Let $U(t, \cdot)=\mathcal{H}_{t} u_0$ be the solution of the heat equation on $[0, \infty)\times \Real^N$ given by \eqref{HeatFlowEqn}.  For any $T_0\in \Real$ and $\mathbf{x}_0\in \Real^N$:
	\begin{enumerate}
		\item $\lim_{t\to \infty} \Vert U(t, \cdot)- \mathcal{H}(t-{T}_0, \cdot, \mathbf{x}_0)\Vert_1=0$;
		\item $\lim_{t\to \infty} t^{\frac{N}{2}} \Vert U(t ,\cdot)-\mathcal{H}(t-T_0, \cdot, \mathbf{x}_0) \Vert_\infty =0$.
	\end{enumerate}
\end{prop}

\section{Properties of Virtual Entropy}
In this section we establish some properties of of the virtual entropy.  First, we recall that finite Colding-Minicozzi entropy is equivalent to uniform bounds on area ratios.
\begin{lem} \label{CMAreaRatLem}
  For $\Sigma\subset \Real^N$ an $n$-dimensional proper submanifold
  $$
  |\Sigma\cap B_R(p)|\leq e^{\pi} \lambda_{CM}^n[\Sigma] R^n.
  $$
 Conversely, suppose there is a $\Theta\geq 1$ so, for all $p$ and $R>0$,
 $$
 |\Sigma\cap B_R(p)|\leq \Theta \omega_n R^n,
 $$
 where $\omega_n$ is the volume of the unit ball in $\Real^n$, then
 $$
 \lambda_{CM}^n[\Sigma]\leq \Theta.
 $$
\end{lem}
\begin{proof}
Let $T=(4\pi)^{-1}R^2$, one has
$$
R^{-n} e^{-\pi}1_{B_R(p)} \leq (4\pi T)^{-\frac{n}{2}} e^{-\frac{|\mathbf{x}-\mathbf{x}(p)|^2}{4 T}}
$$
Hence,
$$
 |\Sigma\cap B_R(p)| =\int_{B_R(p)\cap \Sigma} d\mathcal{H}^n \leq  e^\pi \lambda_{CM}^n[\Sigma] R^n.
 $$
This establishes the first claim.  

For the second we observe that arbitrary scalings and translations of $\Sigma$ satisfy the same area ratio bounds.  Hence, it is enough to check using the co-area formula that
\begin{align*}
\int_{\Sigma} e^{-\frac{|\mathbf{x}|^2}{4}} d\mathcal{H}^n&=\int_0^\infty e^{-\frac{R^2}{4}} \frac{d}{dR} |\Sigma\cap B_R| dR  = \int_0^\infty\frac{R}{2} e^{-\frac{R^2}{4}} |\Sigma\cap B_R| dR\\
&\leq \frac{\Theta \omega_n}{2} \int_0^\infty R^{n+1} e^{-\frac{R^2}{4}}  dR= \Theta (4\pi)^{\frac{n}{2}}.
\end{align*}
Where the final equality can be verified by computing on the flat $n$-plane.  It follows that
$$
\lambda_{CM}^n[\Sigma]\leq \Theta.
$$ 

\end{proof}
We next establish that the virtual and Colding-Minicozzi entropies are comparable.
\begin{prop}\label{VECMBound}
	Let $\Omega\subset \Real^N$ be a non-empty open subset and $\Sigma \subset \Real^N$ an $n$-dimensional submanifold.   There is a constant $C_3=C_3(N)\geq 1$ so that if $f\in \mathcal{VT}_1^+(T)$, $T>0$,  then
$$
(4\pi T)^{\frac{N-n}{2}} \int_{\Sigma\cap \Omega} f d\mathcal{H}^n \leq C_3  \lambda_{CM}^n[\Sigma]  \int_{\mathcal{T}_{2\sqrt{T}}(\Omega)} f d\mathcal{L}^N.
$$	
where $\mathcal{T}_{r} (\Omega)=\bigcup_{p\in \Omega} B_{r}(p)$.
In particular, 
	$$
\lambda_{CM}^n[\Sigma]\leq	\lambda_{V}^n[\Sigma]\leq C_3 \lambda_{CM}^n[\Sigma].
	$$
\end{prop}
\begin{proof}
It follows from Lemma \ref{L12PWEstLem} that for any $q\in \Real^N$
\begin{align*}
\int_{\Sigma\cap B_{\sqrt{T}}(q)} fd\mathcal{H}^n&\leq C_0 \frac{|\Sigma\cap B_{\sqrt{T}}(q)|}{ \omega_N T^{\frac{N}{2}}} \int_{B_{2\sqrt{T}}(q)} f d\mathcal{L}^N\\
&\leq \frac{C_0e^\pi}{\omega_N T^{\frac{N-n}{2}} } \lambda_{CM}^n[\Sigma]  \int_{B_{2\sqrt{T}}(q)} f d\mathcal{L}^N
\end{align*}
where the second inequality used Lemma \ref{CMAreaRatLem} and $f>0$.

Pick a maximal set of points $\set{q_i}$ so $q_i\in \Sigma\cap \Omega$ are extrinsic distance $\geq \sqrt{T}$ and so every point in $\Sigma\cap \Omega$ is distance $\leq \sqrt{T}$ from some $q_i$. The doubling property of $\Real^N$ and choice of points implies there is a constant $D(N)>0$ so for any $q_i$ the set of points $\set{q_j: q_j \in B_{4\sqrt{T}}(q_i)}$ has at most $D(N)$ points.

It follows that
$$
\int_{\Sigma\cap \Omega} fd\mathcal{H}^n \leq \sum_{i=1}^\infty \int_{\Sigma\cap \Omega \cap  B_{\sqrt{T}}(q_i)}  fd\mathcal{H}^n \leq \sum_{i=1}^\infty \frac{C_0 e^\pi}{\omega_N T^{\frac{N-n}{2}} } \lambda_{CM}^n[\Sigma]  \int_{B_{2\sqrt{T}}(q_i)} f d\mathcal{L}^N.
$$
Hence, as $f>0$,
$$
\int_{\Sigma\cap \Omega} fd\mathcal{H}^n \leq \frac{C_0 D(N) e^\pi}{\omega_N T^{\frac{N-n}{2}} } \lambda_{CM}^n[\Sigma]  \int_{\mathcal{T}_{2\sqrt{T}}(\Omega)} f d\mathcal{L}^N.
$$
The first claim follows with
$$
C_3=C_3(N)= (4\pi)^{\frac{N}{2}} \omega_N^{-1} e^\pi C_0 D(N)\geq (4\pi)^{\frac{N-n}{2}} \omega_N^{-1} e^\pi C_0 D(N). 
$$

By taking $\Omega=\Real^N$ we see that for all $f\in \mathcal{VT}_1^+(T)$, 
$$
(4\pi T)^{\frac{N-n}{2}} \int_{\Sigma} fd\mathcal{H}^n \leq C_3\lambda_{CM}^n[\Sigma].
$$
As $ \lambda_{CM}^n[\Sigma]\leq \lambda_V^n[\Sigma]$ follows from $\mathcal{G}_1(T)\subset\mathcal{VT}_1^+(T)$, it follows that
$$
 \lambda_{CM}^n[\Sigma]\leq \lambda_{V}^n[\Sigma]\leq C_3\lambda_{CM}^n[\Sigma]
$$
which verifies the claimed bounds and forces $C_3\geq 1$.
\end{proof}
Virtual entropy is invariant under the natural symmetries of $\Real^N$.
\begin{lem}\label{InvariantLem}
Let $\Sigma\subset \Real^N$ be an $n$-dimensional submanifold.  If $\Sigma'$ is obtained from $\Sigma$ by a rigid motion and scaling, then
$$
\lambda^n_V[\Sigma]=\lambda^n_V[\Sigma'].
$$
\end{lem}
\begin{proof}
By Proposition \ref{VECMBound}, $\lambda^n_V[\Sigma]=\infty$ if and only if $\lambda_{CM}^n[\Sigma]=\infty$. In this case, $\lambda_{CM}^n[\Sigma']=\infty $ and the same is true of $\lambda^n_V[\Sigma']$.  

Hence, we may suppose $\lambda_{CM}^n[\Sigma], \lambda_{CM}^n[\Sigma']<\infty$.
If $R:\Real^N\to \Real^N$ is a rigid motion, then $f\in \mathcal{VT}_1^+(T)$ if and only if $f\circ R \in \mathcal{VT}_1^+(T)$ and
$$
\int_{\Sigma} f\circ R d\mathcal{H}^n= \int_{R(\Sigma)} f d\mathcal{H}^n.
$$
Likewise, if $S_\rho(\mathbf{x})=\rho \mathbf{x}$ is dilation by $\rho>0$, then
$
f\in \mathcal{VT}_1^+(T)$ if and only if $ \rho^N f\circ S_\rho\in \mathcal{VT}_1^+(\rho^{-2} T)$ and
$$
\int_{\Sigma} f\circ S_\rho d\mathcal{H}^n= \rho^{-n} \int_{S_\rho(\Sigma)} f d\mathcal{H}^n.
$$
These observations and the definition of virtual entropy conclude the proof.
\end{proof}

Following \cite{Ecker2001}, we reinterpret the monotonicity property of \cite{hamiltonMonotonicityFormulasParabolic1993} and take care to establish it for non-closed submanifolds.
\begin{prop}\label{MonProp}
		Let $t\mapsto \Sigma_t\subset \Real^N, t\in [0, T_0)$ be a flow of $n$-dimensional proper submanifolds with $\lambda_{CM}^n[\Sigma_t]\leq \Lambda<\infty$ for $t\in [0, T_0)$.  Fix $[T_1,T_2]\subset [0, T_0)$ and $\tau>0$. If $W:[T_1,T_2]\times \Real^N\to \Real$ satisfies 
		\begin{enumerate}
			\item $\left(\frac{\partial}{\partial t}+\Delta_{\Real^N}\right) W=0$;
			\item  $W(t, \cdot)\in \mathcal{VT}_1^+(T_2+\tau-t)$, $t\in [T_1, T_2]$,
		\end{enumerate}
		then
$$
(4\pi \tau)^{\frac{N-n}{2}} \int_{\Sigma_{T_2}} W(T_2, \cdot) d\mathcal{H}^n \leq 		(4\pi (T_2-T_1+\tau))^{\frac{N-n}{2}} \int_{\Sigma_{T_1}} W(T_1, \cdot) d\mathcal{H}^n. 		
$$
\end{prop}
\begin{proof}
Fix $R\geq 1$ and let $\phi_R\geq 0$ be a smooth cutoff function that is identically $1$ in $B_R$ and supported in $B_{2R}$ and so 
$$
\phi_R+R|\nabla_\Real \phi_R|+R^2 |\nabla^2_\Real \phi_R|\leq \kappa
$$
where $\kappa=\kappa(N)>1$.  Set 
$$
V(t, \mathbf{x})= (4\pi (T_2+\tau-t))^{\frac{N-n}{2}}W(t, \mathbf{x})>0.
$$

Using the notation of \cite[Section 3.14]{Ecker2001}, for $t\in (T_1, T_2)$ if we set
$$
F_R(t)=\int_{\Sigma_t} \phi_R(\cdot) V(t, \cdot) d\mathcal{H}^n,
$$
then, following \cite[Section 3.14]{Ecker2001},
$$
F_R'(t)=\int_{\Sigma_t}\left( \left(\frac{d}{dt}-\Delta_{\Sigma_t}\right) \phi_R-\left|\mathbf{H}_{\Sigma_t}+\frac{\nabla^\perp_{\Sigma_t} V}{V}\right|^2 \phi\right) V +\phi_R Q_V  d\mathcal{H}^n
$$
where
$$
Q_V=Q_V(t, \mathbf{x}, N_p \Sigma_t)=\left(\frac{\partial }{\partial t}+\Delta_\Real\right) V-\sum_{i=1}^{N-n}\left( \nabla^2_{\Real} V (E_i, E_i)-\frac{(E_i \cdot \nabla_\Real V)^2}{V} \right)
$$
and $E_1, \ldots, E_{N-n}$ is a choice of orthonormal frame of the normal bundle.  One computes
\begin{align*}
Q_V(t, \mathbf{x}, N_p \Sigma_t)&=- V \sum_{i=1}^{N-n}\left( \nabla_\Real^2\log V (E_i,E_i)+\frac{1}{2 (T_2+\tau-t)} \right)\\
&=- V \sum_{i=1}^{N-n}\left( \nabla_\Real^2\log W (E_i,E_i)+\frac{g_{\Real}(E_i,E_i)}{2 (T_2+\tau-t)} \right).
\end{align*}
In particular, as $W(t, \cdot)\in \mathcal{VT}_1^+(T_2+\tau-t)$  and $V>0$, 
$$
Q_V\leq 0.
$$
Using Proposition \ref{VECMBound},  it follows that
$$
F_R(T_2)-F_R(T_1) \leq \kappa R^{-1} \int_{T_1}^{T_2} \int_{\Sigma_t}  V d\mathcal{H}^n dt\leq C_3 \kappa  \Lambda R^{-1}   (T_2-T_1).
$$
Hence, sending $R\to \infty$ yields
$$
\int_{\Sigma_{T_2}} V d\mathcal{H}^n\leq \int_{\Sigma_{T_1}} V d\mathcal{H}^n.
$$
Expressing $V$ in terms of $W$ completes the proof.
\end{proof}

We can now prove the claimed monotonicity of virtual entropy, i.e., Theorem \ref{LCMonThm}
\begin{proof}[Proof of Theorem \ref{LCMonThm}]
  If $\lambda_{V}^n[\Sigma_0]<\infty$, then the same is true of $\lambda_{CM}^n[\Sigma_0]$ and so the monotonicity of Colding-Minicozzi entropy ensures this for each $\Sigma_t$. In particular,  $\lambda_V^n[\Sigma_t]$ is also finite.  Hence,  for any $T_2\in (0, T)$ and $\epsilon>0$, there is a $\tau_\epsilon>0$ and $u_{T_2,\epsilon}\in\mathcal{VT}_1^+(\tau_\epsilon)$ so
  $$
  \infty>\lambda_V^n[\Sigma_{T_2}]\geq  (4\pi \tau_\epsilon)^{\frac{N-n}{2}} \int_{\Sigma_{T_2}} u_{T_2, \epsilon} d\mathcal{H}^n\geq \lambda_{V}^n[\Sigma_{T_2}]-\epsilon.
  $$
 Let $U_{T_2, \epsilon}(t, \cdot)=\mathcal{H}_t u_{T_2,\epsilon}$ so
 $$
 U_{T_2,\epsilon}(t,\cdot) \in \mathcal{VT}_1^+(\tau_\epsilon+t).
 $$
 Fix a $T_1\in [0, T_2)$ and, for $t\in [T_1, T_2]$, let
 $$
 W(t, \cdot)= U(T_2 -t, \cdot).
 $$
 One readily checks $W$ satisfies the hypotheses of Proposition \ref{MonProp}.  Hence, 
 \begin{align*}
  (4\pi \tau_\epsilon)^{\frac{N-n}{2}} \int_{\Sigma_{T_2}} W(T_2, \cdot) d\mathcal{H}^n \leq (4\pi (T_2-T_1+\tau_\epsilon))^{\frac{N-n}{2}} \int_{\Sigma_{T_1}} W(T_1, \cdot) d\mathcal{H}^n.
 \end{align*}
   It follows that
 \begin{align*}
    \lambda_{V}^n[\Sigma_{T_2}]-\epsilon&\leq   (4\pi \tau_\epsilon)^{\frac{N-n}{2}} \int_{\Sigma_{T_2}} u_{T_2, \epsilon} d\mathcal{H}^n=(4\pi \tau_\epsilon)^{\frac{N-n}{2}} \int_{\Sigma_{T_2}} W(T_2, \cdot) d\mathcal{H}^n \\
    &\leq (4\pi (T_2-T_1+\tau_\epsilon))^{\frac{N-n}{2}} \int_{\Sigma_{T_1}} W(T_1, \cdot) d\mathcal{H}^n\leq \lambda_V^n[\Sigma_{T_1}].
\end{align*}
Where we used that $W(T_1, \cdot )\in \mathcal{VT}_1^+(T_2-T_1+\tau_\epsilon)$ for the final inequality. As $\epsilon>0$ and $0\leq T_1<T_2<T$ are arbitrary the result follows.
    \end{proof}

Using this monotonicity and an observation from \cite{magman08} we obtain precise estimates on the integrals in the definition of $\lambda_{V}^n$ in terms of the Colding-Minicozzi entropy.
\begin{lem}\label{VTBoundCMLem}
	For $T>0$, let $t\in [-T, 0]\mapsto \Sigma_t\subset \Real^N$ be the mean curvature flow of $n$-dimensional proper submanifolds with $\lambda_{V}^n[\Sigma_t]<\infty$.  If $\tau_0>0$ and $u_0\in \mathcal{VT}_1^+(\tau_0)$, then 
	$$
	(4\pi \tau_0)^{\frac{N-n}{2}} \int_{\Sigma_0} u_0 d\mathcal{H}^n\leq ( 1+T^{-1}\tau_0)^{\frac{N-n}{2}}  \lambda_{CM}^n[\Sigma_{-T}].
	$$
\end{lem}
\begin{proof}
	Set $U(t, \cdot)=\mathcal{H}_t u_0$. 
	By Proposition \ref{MonProp}, 
	\begin{align*}
		(4\pi \tau_0)^{\frac{N-n}{2}} \int_{\Sigma_0} u_0 d\mathcal{H}^n
		&\leq (4\pi (\tau_0+T))^{\frac{N-n}{2}} \int_{\Sigma_{-T}} U(T, \cdot) d\mathcal{H}^n.
	\end{align*} 
	By \eqref{HeatFlowEqn},  Fubini's theorem,  the definition of Colding-Minicozzi entropy and fact that $u_0$ has mass one yields
	\begin{align*}
		\int_{\Sigma_{-T}} U(T, \cdot)  d\mathcal{H}^n &=   \int_{\Sigma_{-T}} \int_{\Real^N} u_0(\mathbf{y})\mathcal{H}(T, \mathbf{x}, \mathbf{y}) d\mathcal{L}^N(\mathbf{y}) d\mathcal{H}^n(\mathbf{x}) \\
		&= \int_{\Real^N} u_0(\mathbf{y}) \int_{\Sigma_{-T}} \mathcal{H}(T, \mathbf{x}, \mathbf{y})d\mathcal{H}^n(\mathbf{x}) d\mathcal{L}^N(\mathbf{y} )\\
		&\leq (4\pi T)^{\frac{n-N}{2}} \lambda_{CM}^n[\Sigma_{-T}] \int_{\Real^N} u_0(\mathbf{y})  d\mathcal{L}^N(\mathbf{y})\\
		& = (4\pi T)^{\frac{n-N}{2}} \lambda_{CM}^n[\Sigma_{-T}] .
	\end{align*}
 The claimed inequality follows immediately.
\end{proof}

We now establish \eqref{SelfShrinkEntEqualEqn}.
\begin{thm}\label{SSVTCMThm}
  For a $n$-dimensional proper self-shrinker, $\Sigma$, $\lambda_{CM}^n[\Sigma]=\lambda_{V}^n[\Sigma]$.
\end{thm}
\begin{proof}
 By Proposition \ref{VECMBound}, $\lambda_{V}^n[\Sigma]=\infty$ if and only if $\lambda_{CM}^n[\Sigma]=\infty$ and so we may assume both quantities are finite. The inequality $\lambda_{CM}^n[\Sigma]\leq \lambda_V^n[\Sigma]$ is shown in Proposition \ref{VECMBound}.

 Fix $\tau>0$ and chose any $u_0\in \mathcal{VT}_1^+(\tau)$.  
  Let $t\in (-\infty, 0]\mapsto \Sigma_t=\sqrt{1-t}\Sigma$ be the self-similar flow associated to $\Sigma$ shifted in time. For fixed $t<0$, it follows from Lemma \ref{VTBoundCMLem}, the definition of $\Sigma_t$ and the scaling invariance of Colding-Minicozzi entropy that
  	$$
  (4\pi \tau_0)^{\frac{N-n}{2}} \int_{\Sigma} u_0 d\mathcal{H}^n\leq ( 1+(-t)^{-1}\tau_0)^{\frac{N-n}{2}}  \lambda_{CM}^n[\Sigma].
  $$
Taking $t\to -\infty$ implies
$$
 (4\pi \tau_0)^{\frac{N-n}{2}} \int_{\Sigma} u_0 d\mathcal{H}^n\leq  \lambda_{CM}^n[\Sigma].
$$  
As $\tau_0>0$ and $u_0\in \mathcal{VT}_1^+(\tau_0)$ were arbitrary, $\lambda_V^n[\Sigma]\leq \lambda_{CM}^n[\Sigma]$ completing the proof.
\end{proof}
\begin{rem}
	This result can also be proved using Proposition \ref{LongTimeHeatProp}.  This provides a  heuristic explanation for its validity for self-shrinkers.
\end{rem}

\section{Li-Yau Conformal volume and Stable Conformal Volume}

For $\rho>0$, recall we defined certain weights on $\Real^N$ by
$$
\WW{ n}{\rho}(\mathbf{x})=\frac{\rho^n}{(1+\frac{1}{4}\rho^2|\mathbf{x}|^2)^n}\in C^\infty(\Real^N).
$$
Notice that $\WW{ n}{\rho}\in L^1(\Real^N)$ if and only if $2n>N$.
When $n=N\geq 1$, this is the density on $\Real^N$ from the volume form of the unit $N$-sphere $\mathbb{S}^N\subset \Real^{N+1}$ induced by the stereographic projection map $\mathcal{S}$ of \eqref{SteroEqn} precomposed with the scaling $\mathbf{x}\mapsto \rho \mathbf{x}$;  conjugating with $\mathcal{S}$ gives a M\"{o}bius  transformation on $\mathbb{S}^N$.  Likewise,  Euclidean translations correspond to certain M\"{o}bius  transformations that fix $\infty$, and the corresponding volume forms are given by the translated weights $\WW{N}{\rho}(\cdot -\mathbf{x}_0)$.   

When $n\leq N$, these are the densities induced on $n$-dimensional submanifolds by $\mathcal{S}^*g_{\mathbb{S}}$ by same transformations.  Thus  for the $n$-dimensional affine plane $\Real^n\subset \Real^N$ and $\mathbf{x}_0\in \Real^n$,
$$
|\mathbb{S}^n|=\int_{\Real^n}\WW{ n}{\rho}(\cdot -\mathbf{x}_0)d\mathcal{H}^n \mbox{ and } |\mathbb{S}^N|=\int_{\Real^N} \WW{ N}{\rho}(\cdot -\mathbf{x}_0) d \mathcal{L}^N.
$$
Likewise, when $\Sigma\subset \Real^N$ a general $n$-dimensional submanifold, 
$$
\lambda_{LY}^n[\Sigma]=\sup_{\mathbf{x}_0\in \Real^N, \rho>0} \int_{\Sigma} \WW{ n}{\rho}(\mathbf{x}-\mathbf{x_0}),
$$
is equal to the Li-Yau $N$-conformal volume of $\mathcal{S}(\Sigma)\subset \mathbb{S}^N$.    In what follows it will be convenient to introduce the following modified family of weights for $n, M, \rho>0$
$$
\hat{W}_{M, \rho}^n(\mathbf{x})= (4\pi \rho )^{-\frac{n}{2}} (M \rho)^{\frac{M}{2}} \WW{M}{(M\rho)^{-\frac{1}{2}}}(\mathbf{x})= (4\pi \rho )^{-\frac{n}{2}}\frac{1}{(1+\frac{1}{4 M \rho} |\mathbf{x}|^2)^M}
$$

We observe the following elementary integration result.  
\begin{lem}  \label{IterateLem}
	Suppose that $\Sigma \subset \Real^N$ is a $n$-dimensional submanifold.  	For $m, M\in \mathbb{N}$, and $\mathbf{x}_0=(\mathbf{y}_0, \mathbf{z}_0)\in \Real^N\times \Real^{2m}$,
	$$
	\int_{\Sigma\times \Real^{2m}} \WW{ M+2m}{ \rho}(\mathbf{x}-\mathbf{x}_0) d\mathcal{H}^{n+2m}= \frac{C_{M, m}}{\rho^m} \int_\Sigma \WW{ M+m}{ \rho}(\mathbf{y}-\mathbf{y}_0) d\mathcal{H}^n
	$$
	where $\Sigma\times \Real^{2m}\subset \Real^N\times \Real^{2m}=\Real^{N+2m}$ and
	$$
	C_{M,m}=\frac{(4\pi)^m (M+m-1)!}{(M+2m-1)!}=\frac{2^m |\mathbb{S}^{M+2m}||\mathbb{S}^{M+2m-1}|}{|\mathbb{S}^{M+m}||\mathbb{S}^{M+m-1}|}.
	$$
	As a consequence,
    $$
   \bar{\lambda}_{LY}^{n+2m} [\Sigma]= \hat{C}_{n,m} \sup_{\rho>0, \mathbf{y}_0\in \Real^N}\int_{\Sigma} \hat{W}_{n+m, \rho}^n(\mathbf{y}-\mathbf{y}_0)d\mathcal{H}^n, 
    $$
	where
	$$
	\hat{C}_{n,m}=\left( \frac{4\pi}{n+m}\right)^{\frac{n}{2}} |\mathbb{S}^{n+2m}|^{-1} C_{n,m}.
	$$

\end{lem}
\begin{proof}
	Let us denote the coordinates on $\Real^{N+2}=\Real^N\times \Real^{2}$ by $
	\mathbf{x}=(\mathbf{y}, \mathbf{z})$. Write $\mathbf{x}_0=(\mathbf{y}_0, \mathbf{z}_0)$.  By replacing $\Sigma$ with $\Sigma-(\mathbf{y}_0, \mathbf{0})$ we may, without loss of generality, assume  that  $\mathbf{y}_0=\mathbf{0}$.
 By Fubini's theorem, 
	\begin{align*}
	\int_{\Sigma\times \Real^2} &\WW{ M+2}{ \rho} (\mathbf{x}-\mathbf{x}_0)d\mathcal{H}^{n+2} = \int_\Sigma \int_{\Real^{2}} \frac{\rho^{M+2}}{(1+\frac{\rho^2}{4}|\mathbf{y}|^2+\frac{\rho^2}{4}|\mathbf{z}-\mathbf{z}_0|^2)^{M+2}} d\mathcal{L}^2(\mathbf{z})d\mathcal{L}^n(\mathbf{y})\\
	&= \int_\Sigma \frac{\rho^{M+2}}{(1+\frac{\rho^2}{4} |\mathbf{y}|^2)^{M+2}} \int_{\Real^2} \left(1+\frac{\rho^2}{4+{\rho^2} |\mathbf{y}|^2} |\mathbf{z}-\mathbf{z}_0|^2\right)^{-M-2}d\mathcal{L}^2(\mathbf{z}) d\mathcal{H}^n(\mathbf{y})\\
	&= \frac{4\pi}{M+1} \frac{1}{\rho} \int_{\Sigma} \WW{ M+1}{ \rho} (\mathbf{y})d\mathcal{H}^n(\mathbf{y}).
	\end{align*}
	Where we used  polar coordinates and translation invariance to  evaluate 
	$$
 \int_{\Real^2} \left(1+\frac{\rho^2}{4+{\rho^2} |\mathbf{y}|^2} |\mathbf{z}-\mathbf{z}_0|^2\right)^{-M-2}d\mathcal{L}^2(\mathbf{z}) =
 %2\pi\int_0^\infty (1+\frac{\rho^2}{4+\rho^2 |\mathbf{y}|^2} r^2)^{-M-2} r dr = 
 \frac{4\pi}{M+1}\frac{1+\frac{\rho^2}{4} |\mathbf{y}|^2}{\rho^2}.
	$$
	Iterating this $m$ times  verifies the first equality.  The expression for the $C_{M,m}$  follows from the recursion identity $|\mathbb{S}^{M+2}|=\frac{2\pi}{M+1} |\mathbb{S}^M|$.

	To conclude the proof,  write $\mathbf{x}_0\in \Real^{N+2m}$ as $\mathbf{x}_0=(\mathbf{y}_0, \mathbf{z}_0)\in \Real^N\times \Real^{2m}$.  By translation invariance,  a simple change of variables and fact that $\rho>0$ if and only if $((n+m)\rho)^{-\frac{1}{2}}>0$, one has
\begin{align*}
	\lambda_{LY}^{n+2m}&[\Sigma\times \Real^{2m}]=\sup_{\mathbf{x}_0\in \Real^{N+2m}, \rho>0} \int_{\Sigma\times \Real^{2m}} \WW{n+2m}{ \rho}(\mathbf{x}-\mathbf{x}_0) d\mathcal{H}^{n+2m}\\
 &= \sup_{\mathbf{y}_0\in \Real^{N}, \rho>0} \int_{\Sigma\times \Real^{2m}} \WW{ n+2m}{ ((n+m)\rho)^{-\frac{1}{2}}}(\mathbf{x}-(\mathbf{y}_0, \mathbf{0})) d\mathcal{H}^{n+2m}\\
 &=\sup_{\rho>0} (4\pi \rho)^{\frac{n}{2}} ((n+m) \rho)^{-\frac{n}{2}} C_{n,m}  \sup_{\mathbf{y}_0\in \Real^N} \int_{\Sigma} \hat{W}_{n+m,\rho}^n(\mathbf{y}-\mathbf{y}_0)d\mathcal{H}^n.
\end{align*}
Where the last equality follow from the first part of the proof and the definition of $\hat{W}^n_{n+m, \rho}$.  The claim follows immediately.

\end{proof}

\begin{prop}\label{StableLYIterateProp}
	Let $\Sigma\subset \Real^N$ be an $n$-dimensional submanifold, possibly with boundary.  If $\Sigma$ has a finite conformal volume that is realized, i.e., there are $\mathbf{x}_0\in\Real^N$ and $\rho_0>0$ so
	$$
	\lambda_{LY}^n[\Sigma]=\int_\Sigma\WW{ n}{\rho_0}(\mathbf{x}-\mathbf{x}_0) d\mathcal{H}^n<\infty,
	$$
	then
	$$
	\bar{\lambda}_{LY}^n[\Sigma]\leq \bar{\lambda}_{LY}^{n+2}[\Sigma\times \Real^2]\leq 2 \bar{\lambda}_{LY}^{n}[\Sigma]<\infty.
	$$
	As a consequence,  if $\Sigma\subset \Real^N$ is an $n$-dimensional compact submanifold, possibly with boundary, then, for all $m\in \mathbb{N}$, $m\geq 0$,
	$$
	\bar{\lambda}_{LY}^{n+2m}[\Sigma\times \Real^{2m}]\leq \bar{\lambda}_{LY}^{n+2m+2}[\Sigma\times \Real^{2m+2}] \leq \bar{\lambda}^\infty_{LY}[\Sigma].
	$$
	Moreover, either $\Sigma$ is contained in an affine $n$-dimensional plane or
	$$
	1< \lambda_{LY}^{n+2}[\Sigma\times \Real^2].
	$$
\end{prop}
\begin{rem}
	The affine $n$-dimensional plane $\Real^n \subset \Real^N$ satisfies the hypotheses of the proposition and one computes that
	$$
	1=\bar{\lambda}_{LY}^n[\Real^n]=\bar{\lambda}_{LY}^{n+2}[\Real^{n}\times \Real^2]=\bar{\lambda}_{LY}^{n+2}[\Real^{n+2}].
	$$
	In contrast, round spheres,  $\mathbb{S}^n\subset \Real^N$, give
	$$
	1=\bar{\lambda}_{LY}^n[\mathbb{S}^n]<\bar{\lambda}_{LY}^{n+2}[\mathbb{S}^n\times \Real^2].
	$$
	Here the second inequality is strict by \cite{bryantSurfacesConformalGeometry1988}.  See Appendix \ref{LoxodromeSec} where examples of equality are obtained for non-flat curves.
\end{rem}
\begin{proof}
	One computes, for fixed $\mathbf{x}_0\in \Real^N$ and $\rho>0$,
	\begin{align*}
		\frac{\partial}{\partial\rho}\WW{ M}{ \rho}(\mathbf{x}-\mathbf{x}_0)&= \frac{M \rho^{M-1}}{(1+\frac{\rho^2}{4}|\mathbf{x}-\mathbf{x}_0|^2)^M}-\frac{\frac{1}{2} M \rho^{M+1} |\mathbf{x}-\mathbf{x}_0|^2}{(1+\frac{\rho^2}{4}|\mathbf{x}-\mathbf{x}_0|^2)^{M+1}}\\
		&=-\frac{M}{\rho} \WW{M}{ \rho}(\mathbf{x}-\mathbf{x}_0)+\frac{2M}{\rho^2} \WW{ M+1}{ \rho} (\mathbf{x}-\mathbf{x}_0).
	\end{align*}
Using the pointwise estimate
\begin{equation}\label{WeightEstEqn}
	\WW{ M+1}{ \rho}(\mathbf{x}-\mathbf{x}_0)\leq \rho (1+\frac{\rho^2}{4}|\mathbf{x}-\mathbf{x}_0|^2)^{-1} \WW{ M}{ \rho}(\mathbf{x}- \mathbf{x}_0) \leq \rho \WW{ M}{\rho}(\mathbf{x}- \mathbf{x}_0),
\end{equation} 
one sees that if $\int_\Sigma\WW{ M}{ \rho}(\mathbf{x}-\mathbf{x}_0)d\mathcal{H}^n<\infty,$
 then differentiation by $\rho$ under the integral sign is justified.  In this case,
		\begin{align*}
		\frac{\partial}{\partial\rho} \int_{\Sigma} \WW{ M}{ \rho}(\mathbf{x}-\mathbf{x}_0) d\mathcal{H}^n=\int_{\Sigma} \frac{2M}{\rho^2}\WW{ M+1}{ \rho}(\mathbf{x}-\mathbf{x}_0)-\frac{M}{\rho }  \WW{ M}{ \rho}(\mathbf{x}-\mathbf{x}_0) d\mathcal{H}^n.
	\end{align*}
Specializing to the case $M=n$, if 
	$$
	\lambda_{LY}^n[\Sigma]=\int_\Sigma \WW{ n}{ \rho_0}(\mathbf{x}-\mathbf{x}_0) d\mathcal{H}^n<\infty,
	$$
	then there is finite maximum at $\rho=\rho_0$ and so
	$$
	0 =\frac{2n}{\rho^2_0}\int_{\Sigma} \WW{ n+1}{\rho_0}(\mathbf{x}-\mathbf{x}_0) d\mathcal{H}^n -\frac{n}{\rho_0 } \int_{\Sigma} \WW{ n}{ \rho_0}(\mathbf{x}-\mathbf{x}_0) d\mathcal{H}^n.
	$$
	It follows from Lemma \ref{IterateLem} that, with $\mathbf{x}_0=(\mathbf{y}_0,\mathbf{0})\in \Real^{N+2}$, 
	\begin{align*}
		\lambda_{LY}^n[\Sigma]&=\int_\Sigma \WW{ n}{\rho_0}(\mathbf{y}-\mathbf{y}_0) d\mathcal{H}^n=\frac{2}{\rho_0} \int_{\Sigma} \WW{n+1}{\rho_0}(\mathbf{y}-\mathbf{y}_0) d\mathcal{H}^n\\
		&= \frac{n+1}{2\pi}  \int_{\Sigma\times \Real^2} \WW{ n+2}{ \rho_0} (\mathbf{x}-\mathbf{x}_0) d\mathcal{H}^{n+2}\leq  \frac{n+1}{2\pi} \lambda_{LY}^{n+2} [\Sigma\times \Real^2].
\end{align*}
	As $|\mathbb{S}^{n+2}|=\frac{2\pi}{n+1} |\mathbb{S}^n|$, it follows that
	$$
	\bar{\lambda}_{LY}^n[\Sigma]\leq  \bar{\lambda}_{LY}^{n+2} [\Sigma\times \Real^2].
	$$
Likewise,  \eqref{WeightEstEqn} and Lemma \ref{IterateLem} imply that when $\mathbf{x}_0=(\mathbf{y}_0, \mathbf{z}_0)\in \Real^N\times \Real^2$, 
	\begin{align*}
		\int_{\Sigma\times \Real^2} \WW{ n+2}{\rho}(\mathbf{x}-\mathbf{x}_0)d\mathcal{H}^{n+2}&= \frac{4\pi}{n+1} \rho^{-1} \int_{\Sigma}\WW{ n+1}{ \rho}(\mathbf{y}-\mathbf{y}_0)d\mathcal{H}^{n}\\
		&\leq \frac{4\pi}{n+1} \int_{\Sigma} \WW{ n}{\rho}(\mathbf{y}- \mathbf{y}_0) d\mathcal{H}^n \leq \frac{4\pi}{n+1}\lambda_{LY}^n[\Sigma]
	\end{align*}
	Taking the sup over $\rho>0$ and $\mathbf{x}_0\in \Real^{N}\times \Real^2$ yields
	$$
	\lambda_{LY}^{n+2}[\Sigma\times \Real^2]\leq \frac{4\pi}{n+1} \lambda_{LY}^n[\Sigma]
	$$
	Which gives claimed bound between the normalized conformal volumes.
	
To complete the proof let $\Sigma_m=\Sigma\times \Real^{2m}\subset \Real^{N+2m}$, $n+2m\geq 1$. We claim exactly one of the following situations holds:
	\begin{enumerate}
			\item $\Sigma$ contained in an affine $n$-plane  and, for all $m\geq 0$, $\bar{\lambda}_{LY}^{n+2m}[\Sigma_m]=\bar{\lambda}_{LY}[\Sigma]=1$;
			\item $n=1$ and,  for all $m\geq 1$, $\bar{\lambda}_{LY}^{1+2m}[\Sigma_m]>\bar{\lambda}_{LY}^1[\Sigma_0]=\bar{\lambda}_{LY}^1[\Sigma]=1$; 
		\item For all $m\geq 0$, $\bar{\lambda}_{LY}^{n+2m}[\Sigma_m]>1$. 
\end{enumerate}
	First of all, by \cite{bryantSurfacesConformalGeometry1988}, one has for all $m\geq 0$,
	$$
	\lambda_{LY}^{n+2m}[\Sigma_m]\geq |\mathbb{S}^{n+2m}| \Rightarrow \bar{\lambda}_{LY}^{n+2m}[\Sigma_m]\geq 1.
	$$
	Moreover, when $n+2m> 1$,  the inequality is strict at a fixed $m$ unless $\Sigma_{m}$ is contained in a $(n+2m)$-dimensional  affine plane. In the latter situation,   $\Sigma$ is contained in an affine $n$-dimensional plane which then implies the same for all $\Sigma_{m}$, $m\geq 0$.   That is, case (1) holds.
	When $n=1$ and $m=0$, it may happen that $\bar{\lambda}_{LY}^1[\Sigma]=1$ and $\Sigma$ is not contained in an affine line.  In this case, the same reasoning implies that, for $m\geq 1$, $\lambda_{LY}^{1+2m}[\Sigma_m]>1=\bar{\lambda}_{LY}^1[\Sigma]$ and we are in case (2).  Any other situation is covered by case (3), establishing the claim.
	
	It is immediate that the remaining conclusions of the proposition hold when case (1) holds for all $n$ and $m$ and they hold in case (2) when $m=0$. Hence, it suffices to suppose $\bar{\lambda}_{LY}^{n+2m}[\Sigma_m]>1$.
Using what has already been established,  we need only show that this condition and fact that $\Sigma$ is compact ensures $\bar{\lambda}_{LY}^{n+2m}[\Sigma_m]$ is realized. 
 
For any fixed $\mathbf{x}_0\in \Real^N\times \Real^{2m}=\Real^{N+2m}$ write $\mathbf{x}_0=(\mathbf{y}_0, \mathbf{z}_0)$. 
	By Lemma \ref{IterateLem}, 
	\begin{align*}
		\int_{\Sigma_m} \WW{ n+2m}{ \rho}(\mathbf{x}-\mathbf{x}_0)d\mathcal{H}^{n+2m}
		&=\frac{C_{n, m}}{\rho^m}\int_{\Sigma} \WW{ n+m}{ \rho}(\mathbf{y}-\mathbf{y}_0) d\mathcal{H}^{n}\leq C_{n, m} |\Sigma |\rho^n<\infty.
	\end{align*}
 Where we also used that $\Sigma$ is compact and of finite area and \eqref{WeightEstEqn} to estimate 
	$$
	\rho^{-m}  \WW{ n+m}{\rho}(\mathbf{y}-\mathbf{y}_0) \leq \WW{ n} {\rho}(\mathbf{y}-\mathbf{y}_0)\leq \rho^n.
	$$
	Thus, there is a uniform $\rho_m^->0$ so, for $0<\rho<\rho_m^-$, 
	\begin{equation*}
	|\mathbb{S}^{n+2m}|^{-1}	\int_{\Sigma_m} \WW{ n+2m}{\rho}(\mathbf{x}-\mathbf{x}_0)d\mathcal{H}^{n+2m}\leq \frac{3}{4}<\bar{\lambda}_{LY}^{n+2m}[\Sigma_m].
	\end{equation*}
	When
\begin{equation}\label{NotCloseEqn}
	\inf_{\mathbf{y}\in \Sigma} |\mathbf{y}-\mathbf{y}_0|> 4,
\end{equation}
iterating \eqref{WeightEstEqn} and combining with Lemma \ref{IterateLem} yields
	$$
	\int_{\Sigma_m} \WW{ n+2m}{\rho}(\mathbf{x}-\mathbf{x}_0)d\mathcal{H}^{n+2m}\leq \frac{C_{n, m}}{\rho^{m}}	\int_{\Sigma} \WW{ n+m}{\rho}(\mathbf{y}-\mathbf{y}_0)d\mathcal{H}^{n}\leq \frac{C_{n, m}}{(1+\rho^2)^{m}} \lambda_{LY}^{n}[\Sigma]
	$$
	and so there is a $\rho_m^+$ so for $\rho>\rho_m^+$ and $\mathbf{y}_0$ not close to $\Sigma$ in sense of \eqref{NotCloseEqn}
	$$
	|\mathbb{S}^{n+2m}|^{-1}\int_{\Sigma_m} \WW{ n+2m}{\rho}(\mathbf{x}-\mathbf{x}_0)d\mathcal{H}^{n+2m}\leq \frac{3}{4}<\bar{\lambda}_{LY}^{n+2m}[\Sigma_m].
	$$
	Finally, as $\Sigma$ is compact, the set of $\mathbf{y}_0$  not satisfying \eqref{NotCloseEqn}, i.e., so $\inf_{\mathbf{y}\in \Sigma} |\mathbf{y}-\mathbf{y}_0|\leq 4$, 
	is compact and so it follows from the asymptotic expansion of \cite{bryantSurfacesConformalGeometry1988} and $\bar{\lambda}_{LY}^{n+2m}[\Sigma']>1$ that, after possibly increasing $\rho_m^+$,  when $\rho>\rho_m^+$ and $\mathbf{y}_0$ does not satisfy \eqref{NotCloseEqn}
	$$
	|\mathbb{S}^{2n+m}|^{-1}\int_{\Sigma_m} \WW{ n+2m}{ \rho}(\mathbf{x}-\mathbf{x}_0)d\mathcal{H}^{n+2m}\leq \frac{1}{2}(1+\bar{\lambda}_{LY}^{n+2m}[\Sigma_m])<\bar{\lambda}_{LY}^{n+2m}[\Sigma_m].
	$$
%	Hence, , for $\rho>\rho_m^+$, 
%	$$
%	\int_{\Sigma'} \WW{ n+2m}{\rho}(\mathbf{x}-\mathbf{x}_0)d\mathcal{H}^{n+2m}< \lambda_{LY}^{n+2m}[\Sigma'].
%	$$
Thus, by the translation invariance in the $\mathbb{R}^{2m}$ factor,  there is a compact $K_m\subset \Real^{N}$ so
	$$
	\lambda_{LY}^{n+2m}[\Sigma_m]=\sup_{\rho_m^-\leq \rho \leq \rho_m^+,\mathbf{y}_0\in K_m} \int_{\Sigma_m}\WW{ n+2m}{ \rho}(\mathbf{x}-(\mathbf{y}_0, \mathbf{0}) ) d\mathcal{H}^{n+2m}.
	$$
	As this is the supremum of a continuous function over a compact set, it is a maximum and so the conformal volume is achieved.  This completes the proof.
\end{proof}
We can now prove Proposition \ref{StabConfVolMonDimProp}
\begin{proof}[Proof of Proposition \ref{StabConfVolMonDimProp}]
	When $\bar{\lambda}_{LY}^n[\Sigma]<\infty$, for any $\epsilon>0$ there are $\mathbf{x}_0, \rho_0$ so
	$$
\infty>\lambda_{LY}^n[\Sigma]\geq 	\int_{\Sigma} \WW{ n}{\rho_0} (\mathbf{x}-\mathbf{x}_0) d\mathcal{H}^n \geq \lambda_{LY}[\Sigma]-\epsilon.
	$$
Pick $R_\epsilon>0$ large enough so that $\Sigma_\epsilon =\bar{B}_{R_\epsilon}\cap \Sigma$ satisfies
	$$
	\lambda_{LY}^n[\Sigma_\epsilon]\geq  \int_{\Sigma_\epsilon}  \WW{ n}{\rho_0}(\mathbf{x}-\mathbf{x}_0) d\mathcal{H}^n\geq \lambda_{LY}^n[\Sigma]-2\epsilon.
	$$
	By slightly increasing $R_{\epsilon}$ we may also assume $\Sigma_\epsilon$ is a compact manifold with boundary.
	As $\Sigma_\epsilon$ is compact, we can apply Proposition \ref{StableLYIterateProp} to obtain that
	$$
	\bar{\lambda}_{LY}^{n+2m}[\Sigma_\epsilon\times \Real^{2m}]\geq \bar{\lambda}_{LY}^n[\Sigma_\epsilon]\geq \bar{\lambda}_{LY}^n[\Sigma]-2|\mathbb{S}^n|^{-1}\epsilon.
	$$
	As $\Sigma_\epsilon\times \Real^{2m}\subset \Sigma \times \Real^{2m}$, 
	$$
	\bar{\lambda}_{LY}^{n+2m}[\Sigma \times \Real^{2m}]\geq  \bar{\lambda}_{LY}^{n+2m}[\Sigma_\epsilon\times \Real^{2m}]\geq\bar{\lambda}_{LY}^n[\Sigma]-2|\mathbb{S}^n|^{-1}\epsilon.
	$$
	The choice of $\epsilon>0$ was arbitrary so we conclude that
	$$
	\bar{\lambda}_{LY}^{n+2m}[\Sigma \times \Real^{2m}]\geq \bar{\lambda}_{LY}^n[\Sigma].
	$$
	
	When $\bar{\lambda}_{LY}^n[\Sigma]=\infty$, for any $\Lambda>0$ large there are $\mathbf{x}_0$ and $\rho_0$ so
	$$
	\int_{\Sigma} \WW{ n}{\rho}(\mathbf{x}-\mathbf{x}_0)\geq 2\Lambda.
	$$
	Now pick $R_\Lambda>0$ large enough so that  $\Sigma_{\Lambda}=B_{R_\Lambda}\cap \Sigma$ satisfies
	$$
	\lambda_{LY}^n[\Sigma_\Lambda]\geq  \int_{\Sigma_\Lambda}  \WW{ n}{\rho_0} (\mathbf{x}-\mathbf{x}_0) d\mathcal{H}^n\geq \Lambda.
	$$
	We may again increase $R_\Lambda$ slightly to ensure $\Sigma_\Lambda$ is a compact manifold with boundary. Proposition \ref{StableLYIterateProp} implies
	$$
	\bar{\lambda}_{LY}^{n+2m} [\Sigma\times \Real^{2m}]\geq  \bar{\lambda}_{LY}^{n+2m} [\Sigma_\Lambda\times \Real^{2m}]\geq \bar{\lambda}_{LY}^{n} [\Sigma_\Lambda]\geq |\mathbb{S}^n|^{-1} \Lambda.
	$$
	As $\Lambda$ is arbitrary, we conclude
	$$
	\bar{\lambda}_{LY}^{n+2m} [\Sigma\times \Real^{2m}]=\infty.
	$$
\end{proof}

We will show,  that in general,
$$
\lambda_{CM}^n[\Sigma]\leq \bar{\lambda}^\infty_{LY}[\Sigma]\leq\lambda_{V}^n[\Sigma].
$$
\begin{prop} \label{ConfVolCMentProp}
One has
$$
(4\pi \rho)^{-\frac{n}{2}} e^{-\frac{|\mathbf{x}|^2}{4 \rho}} =\lim_{m\to \infty}\hat{W}_{n+m, \rho}^n(\mathbf{x})\leq 	\hat{W}_{n+m+1, \rho}^n(\mathbf{x})\leq \hat{W}_{n+m, \rho}^n(\mathbf{x}) \leq \hat{W}_{n, \rho}^n(\mathbf{x}).	$$
Likewise,
$$
\left(\frac{4\pi}{n}\right)^{\frac{n}{2}} |\mathbb{S}^n|^{-1}=\hat{C}_{n,0}< \hat{C}_{n, m}<\hat{C}_{n, m+1}< \lim_{m\to \infty} \hat{C}_{n, m}=1.
$$
As a consequence, for an $n$-dimensional submanifold, $\Sigma \subset \Real^N$, 
$$
\hat{C}_{n, m} \lambda_{CM}^n[\Sigma]\leq \bar{\lambda}_{LY}^{n+2m}[\Sigma\times \Real^{2m}];
$$
$$
	\lambda_{CM}^n[\Sigma]\leq \bar{\lambda}_{LY}^\infty[\Sigma].
$$
\end{prop}
\begin{proof}
	Observe that $\hat{W}^n_{n+m+1, \rho}\leq \hat{W}^n_{n+m, \rho}$ if and only if for $t\geq 0$,
	$$
	-(n+m+1) \log \left(1+\frac{t}{n+m+1} \right)\leq -(n+m) \log (1+\frac{t}{n+m}).
	$$
	Fix $t\geq 0$ and consider the function
	$$
 I_t(s)= -s \log (1+\frac{t}{s}) , s>0.
	$$
	It will suffice to show this is a non-increasing function for $s\geq 1$.
	One has 
	$$I_t'(s)= -\log (1+\frac{t}{s}) +\frac{t}{s}  (1+\frac{t}{s})^{-1}= 1 -\log (1+\frac{t}{s}) - (1+\frac{t}{s})^{-1},$$
	$$
	I_t''(s)= \frac{t}{s^2}  (1+\frac{t}{s})^{-1} -\frac{t}{s^2}  (1+\frac{t}{s})^{-2}= \frac{t^2}{s^3}  (1+\frac{t}{s})^{-2} \geq 0.
	$$
	As $\lim_{s\to \infty} I_t'(s)=0$, it follows that $I_t'(s)\leq 0$ for $s\geq 1$ and $t\geq 0$.  That is,  $I_t(s)$ is non-increasing as claimed.
	
	For fixed $\mathbf{x}\in \Real^N$,  one has the elementary limit
	$$
	\lim_{m\to \infty}  (1+\frac{1}{4(n+m) \rho} |\mathbf{x}|^2)^{-n-m}= e^{-\frac{|\mathbf{x}|^2}{4\rho}}.
	$$
	Hence, the $\hat{W}^n_{n+m, \rho}$ decrease pointwise in $m$ to
	$$
	(4\pi \rho)^{-\frac{n}{2}} e^{-\frac{|\mathbf{x}|^2}{4 \rho}}=\lim_{m\to \infty}\hat{W}_{n+m, \rho}^n(\mathbf{x}).
	$$
	
	With $\Sigma=\Real^n$ the $n$-dimensional affine plane through the origin,  Lemma \ref{IterateLem} yields:
	$$
1=\bar{\lambda}_{LY}^{n+2m} [\Real^n \times \Real^{2m}]= \hat{C}_{n, m} \int_{\Real^n} \hat{W}_{n+m, 1}^n d\mathcal{H}^n=\hat{C}_{n, m}\int_{\Real^n} \hat{W}_{n+m, 1}^n  d\mathcal{L}^n.
$$
As the $\hat{W}_{n+m,1}^n$ are decreasing pointwise in $m$, the $\hat{C}_{n,m}$ are increasing in $m$.  Moreover,
$$
1=\int_{\Real^n} 	(4\pi )^{-\frac{n}{2}} e^{-\frac{|\mathbf{x}|^2}{4 }} d\mathcal{L}^n=\lim_{m\to \infty} \int_{\Real^n} \hat{W}^n_{n+m,1} d\mathcal{L}^n
$$
by the dominated convergence theorem and so $\lim_{m\to \infty} \hat{C}_{n,m}=1$.

Finally, by Lemma \ref{IterateLem} we see that
	\begin{align*}
	\hat{C}_{n,m}	F_{\mathbf{x}_0,t_0}[\Sigma] &= \hat{C}_{n,m}\int_{\Sigma} (4\pi t_0)^{-\frac{n}{2}} e^{-\frac{|\mathbf{x}-\mathbf{x}_0|^2}{4t_0}} d\mathcal{H}^n\\
		&\leq 	\hat{C}_{n,m}\int_{\Sigma} \hat{W}_{n+m, t_0}^n(\mathbf{x}-\mathbf{x}_0) d\mathcal{H}^{n}= \bar{\lambda}_{LY}^{n+2m} [\Sigma \times \Real^{2m}].
	\end{align*}
	By taking the supremum over $\mathbf{x}_0$ and $t_0$ we obtain the first estimate.  The second follows by taking $m\to \infty$.
	
\end{proof}
We record an elementary computation:
\begin{lem}\label{LogConvexComputeLem}
	For $M\geq 0$, 
	\begin{align*}
		\nabla^2_\Real \log \WW{ M}{\rho} & = -\frac{1}{2} M \rho^2 g_{\Real}+\frac{M}{8}  \frac{ \rho^4|\mathbf{x}|^2}{1+\frac{\rho^2}{4} |\mathbf{x}|^2} g_\Real +\frac{M}{16} \frac{\rho^4 d|\mathbf{x}|^2\otimes d |\mathbf{x}|^2 }{(1+\frac{\rho^2}{4} |\mathbf{x}|^2)^2}\geq -\frac{1}{2} M \rho^2 g_{\Real}
	\end{align*}
	with equality at $\mathbf{x}=\mathbf{0}$.
	In particular,
	$$
	\tau(\WW{ M}{ \rho})=\frac{1}{M \rho^2} \mbox{ and } \tau(\hat{W}_{n+m, \rho}^n )=\rho.
	$$
\end{lem}
\begin{proof}
	One directly computes
	$$	\nabla_\Real \log \WW{ M}{ \rho}(\mathbf{x})= -\frac{M \rho^2}{2} \frac{\mathbf{x}}{1+\frac{\rho^2}{4}|\mathbf{x}|^2}; 
	$$
	$$
	\nabla_\Real^2 \log \WW{ M}{ \rho} (\mathbf{x})= -\frac{M \rho^2}{2} \frac{g_{\Real}}{1+\frac{\rho^2}{4}|\mathbf{x}|^2}+\frac{M}{16} \frac{\rho^4 d|\mathbf{x}|^2\otimes d |\mathbf{x}|^2 }{(1+\frac{\rho^2}{4} |\mathbf{x}|^2)^2}.
	$$	
	It is straightforward to verify the claims from this.
\end{proof}

\begin{prop}\label{WeightInVTProp}
For $n,m, N\in \mathbb{N}$ with $m\geq N-n$,
$$
\hat{C}_{N, n-N+m} ( 4\pi \rho)^{\frac{n-N}{2}}\hat{W}_{n+m, \rho}^n \in \mathcal{VT}_1^+(\rho).
$$
In particular, if $\Sigma\subset \Real^N$ is a $n$-dimensional proper submanifold and $m\geq N=n$, then
$$
\mathcal{\lambda}_{V}^n[\Sigma]\geq \hat{C}_{N, n-N+m}  \hat{C}_{n,m}^{-1} \bar{\lambda}_{LY}^{n+2m}[\Sigma\times \Real^{2m}]
$$
$$
\mathcal{\lambda}_{V}^n[\Sigma]\geq\bar{\lambda}^\infty_{LY}[\Sigma].
$$
\end{prop}
\begin{proof}
First observe that by Lemma \ref{LogConvexComputeLem}, $\tau(\hat{W}_{n+m, \rho}^n) =\rho$.  Hence, it suffices to calculate the mass of $\hat{W}_{n+m, \rho}^n$ and scale by it's inverse.  To that end, observe that $m+n \geq N\geq 1$ ensures $\hat{W}_{n+m, \rho}^n\in L^1(\Real^N)$.  By considering the change of variables $\mathbf{x}\mapsto \sqrt{\rho} \mathbf{x}$ one has
$$
\int_{\Real^N} \hat{W}_{n+m, \rho}^n d\mathcal{L}^N= \rho^{\frac{N-n}{2}} \int_{\Real^N}\hat{W}_{n+m, 1}^n  d\mathcal{L}^N.
$$
Define $\alpha_{n,m,N}\in (0, \infty)$ by
$$
\alpha_{n,m,N}^{-1} = \int_{\Real^N}\hat{W}_{n+m, 1}^n  d\mathcal{L}^N
$$
For,  $m'=n-N+m \geq 0$, it is clear that
$$\hat{W}_{n+m, 1}^n= (4\pi)^{\frac{N-n}{2}} \hat{W}_{N+m',1}^N.$$
Hence, for $m\geq N-n$, as $\bar{\lambda}_{LY}^{N+2m'}[\Real^N\times \Real^{2m'}]=1$, Lemma \ref{IterateLem} gives
$$
\int_{\Real^N} \hat{W}_{n+m, 1}^n= (4\pi)^{\frac{N-n}{2}} \int_{\Real^N}\hat{W}_{N+m',1}^N =  (4\pi)^{\frac{N-n}{2}} \hat{C}_{N, m'}^{-1}
$$
That is, for $m\geq N-n$,
$$
\alpha_{n,m,N}=  (4\pi)^{\frac{n-N}{2}} \hat{C}_{N, n-N+m} \mbox{ and } \alpha_{n,m, N} \hat{W}_{n+m,1}^n\in \mathcal{VT}_1^+(\rho).
$$

To show the first inequality for $\Sigma$ set
$$
\tilde{W}_{n+m,\rho}^n= \alpha_{n,m,N} \rho^{\frac{n-N}{2}} \hat{W}_{n+m, \rho}^n=(4\pi)^{\frac{n-N}{2}} \hat{C}_{N, n-N+m}\hat{W}_{n+m, \rho}^n.
$$
Clearly, along with any of its translates one has
$$
\tilde{W}_{n+m,\rho}^n\in \mathcal{VT}_1^+(\rho).
$$
Hence, the definition of virtual entropy gives
\begin{align*}
	\lambda_V^n[\Sigma] &\geq (4\pi \rho)^{\frac{N-n}{2}} \int_{\Sigma} \tilde{W}_{n+m,\rho}^n(\mathbf{x}-\mathbf{x}_0) d\mathcal{H}^n \\
	&= (4\pi )^{\frac{N-n}{2}}\alpha_{n,m,N} \hat{C}_{n,m}^{-1} \hat{C}_{n,m}\int_{\Sigma} \hat{W}_{n+m,\rho}^n(\mathbf{x}-\mathbf{x}_0) d\mathcal{H}^n.
\end{align*}
By Lemma \ref{IterateLem}, taking the supremum over $\rho>0$ and $\mathbf{x}_0\in \mathbb{R}^N$ yields
\begin{equation}\label{LVLYmEqn}
	\lambda_V^n[\Sigma]\geq \hat{C}_{N, n-N+m}  \hat{C}_{n,m}^{-1} \bar{\lambda}_{LY}^{n+2m}[\Sigma\times \Real^{2m}].
\end{equation}
This gives the first inequality for $\Sigma$.  Taking $m\to \infty$ and appealing to Proposition \ref{ConfVolCMentProp} gives the second one.
Note that plugging $\Sigma=\Real^n$ into the \eqref{LVLYmEqn} implies
$$
\hat{C}_{n,m}\geq \hat{C}_{N, n-N+m}.
$$
This gives an alternative way to conclude the second inequality.
\end{proof}

 Theorem \ref{CMLYVTIneqThm} is an immediate consequence of Propositions \ref{ConfVolCMentProp} and \ref{WeightInVTProp}.
\appendix

\section{Loxodromes} \label{LoxodromeSec}

For $\alpha\in [0, \infty)$ consider the rectifiable set:
$$
S_\alpha^0=\set{(e^t\cos \alpha t, e^t\sin \alpha t): t\in \Real}\subset \Real^2.
$$
This is a smooth curve away from $\mathbf{0}$.
When $\alpha>0$ this is a logarithmic spiral (i.e., a special case of a loxodrome or rhumb line) and for $\alpha= 0$ it is the non-negative part of the $x$-axis.    If $R_\theta: \Real^2 \to \Real^2$ is rotation counter clockwise by $\theta$, then we see that for $\alpha>0$, $S^\theta_\alpha=R_{\theta} (S_\alpha^0)= e^{-\alpha^{-1}\theta} \cdot S_\alpha^0$, i.e., scaling $S_\alpha^0$  about the origin is the same as rotating it.  Consider the doubled curve
$$
S_\alpha=S_\alpha^0\cup S_\alpha^\pi
$$
which is symmetric with respect to reflection through the origin.

\begin{prop}\label{LoxodromeConfLenProp}
For any $\alpha\geq 0$ one has,
	$$
	\lambda_{CM}^1[S_\alpha]=\bar{\lambda}_{LY}^\infty[S_\alpha]=\bar{\lambda}_{LY}^1[S_\alpha]=\sqrt{1+\alpha^2}.
	$$
\end{prop}
\begin{rem}
  This shows the existence of multiplicity one curves  that have entropy any value larger than one.  In particular, it suggests some of the results of \cite{chenRigidityStabilitySubmanifolds2021} are optimal.
\end{rem}
\begin{proof}
When $\alpha=0$ this is elementary as $S_0$ is a line. In what follows we suppose $\alpha>0$.
In general,  the parameterization of $S_\alpha^0$ and a change of variables gives
$$
\int_{	S_\alpha} f(|\mathbf{x}|) d\mathcal{H}^1=2\int_{-\infty}^\infty f(e^r) \sqrt{1+\alpha^2}e^r dr=2\sqrt{1+\alpha^2} \int_0^\infty f(u) du.
$$
As scaling $S_\alpha$ is the same as rotating it around the origin, it follows from the change of variables formula and rotational symmetry of the weights that
$$
	\lambda_{CM}^1[S_\alpha]=\sup_{\mathbf{x}_0\in \Real^2} (4\pi)^{-\frac{1}{2}} \int_{S_\alpha} e^{-\frac{|\mathbf{x}-\mathbf{x}_0|^2}{4}} d\mathcal{H}^1.
$$
Hence,
$$
	\lambda_{CM}^1[S_\alpha]\geq  (4\pi)^{-\frac{1}{2}} \int_{S_\alpha} e^{-\frac{|\mathbf{x}|^2}{4}} d\mathcal{H}^1= 2 (4\pi)^{-\frac{1}{2}} \sqrt{1+\alpha^2} \int_0^\infty  e^{-\frac{u^2}{4}} du=\sqrt{1+\alpha^2}.
$$
For the same reason,
$$
	\bar{\lambda}_{LY}^\infty[S_\alpha]\geq 	\bar{\lambda}_{LY}^1[S_\alpha]=(2\pi)^{-1}\sup_{\mathbf{x}_0\in \Real^2} \int_{S_\alpha} \WW{1}{1}(\mathbf{x}-\mathbf{x}_0) d\mathcal{H}^1 \geq \sqrt{1+\alpha^2}.
$$

In order to establish the claim, let us suppose $\mathbf{x}_0=(x_0,y_0)\in\mathbb{S}^1\subset \Real^2$ is a fixed unit vector. 
For $t\in \Real$ and $m\geq 0$ set
$$
L_m^{\alpha, \mathbf{x}_0}(t)= 	\int_{S_\alpha} \WW{ m+1}{1}(\mathbf{x}-t\mathbf{x}_0)d\mathcal{H}^1.
$$
Up to a constant factor, these are the integrals that compute $\bar{\lambda}_{LY}^{1+2m}[S_\alpha\times \Real^{2m}]$.

Using the parameterization of $S_\alpha$ we have
\begin{align*}
L_{m-1}^{\alpha, \mathbf{x}_0}(t)&= \int_{S_\alpha} \frac{1}{(1+\frac{1}{4} |\mathbf{x}-t\mathbf{x}_0|^2)^m} d\mathcal{H}^1\\
&=\int_{-\infty}^\infty   \frac{\sqrt{1+\alpha^2}e^r}{(1+\frac{1}{4} ((e^r \cos(\alpha r)-t x_0)^2+ (e^r \sin (\alpha r)-ty_0)^2))^m} dr\\
&+\int_{-\infty}^\infty   \frac{\sqrt{1+\alpha^2} e^r}{(1+\frac{1}{4} ((e^r\cos(\alpha r)+t x_0)^2+ (e^r \sin (\alpha r)+ty_0)^2))^m} dr\\
&=\sqrt{1+\alpha^2}\int_{-\infty}^\infty \frac{e^r}{(I_1(r)+I_2(r))^m}+\frac{e^r}{(I_1(r)-I_2(r))^m}dr
\end{align*}
where here
$$I_1(r)= 1+ \frac{1}{4} (e^{2r} +t^2) \geq 1 \mbox{ and } I_2(r) =-\frac{1}{2} tx_0 e^r\cos(\alpha r) -\frac{1}{2} ty_0 e^r\sin(\alpha r).$$
By the Cauchy-Schwarz and absorbing inequality, 
$$
I_2(r)^2\leq J_2(r)^2= \frac{1}{4} t^2 e^{2r}\leq \frac{1}{16}( t^2 +e^{2r})^2 \leq I_1(r)^2.
$$
The following inequality holds: for $a^2\leq b^2$, $ x\geq 0$ and an integer $m\geq 1$,
$$
(x-a)^m+(x+a)^m\leq (x-b)^m+(x+b)^m.
$$
Using this we obtain the estimate
\begin{align*}
\frac{L_{m-1}^{\alpha, \mathbf{x}_0}(t)}{\sqrt{1+\alpha^2}}&=\int_{-\infty}^\infty \frac{(I_1(r)-I_2(r))^m+(I_1(r)+I_2(r))^m}{(I_1(r)^2-I_2(r)^2)^m} e^r dr\\
&\leq \int_{-\infty}^\infty \frac{(I_1(r)-J_2(r))^m+(I_1(r)+J_2(r))^m}{(I_1(r)^2-J_2(r)^2)^m} e^r dr\\
&=\int_{-\infty}^\infty  \frac{e^r}{ (I_1(r)+J_2(r))^m}+ \frac{e^r}{ (I_1(r)-J_2(r))^m} dr\\
%&=\int_{-\infty}^\infty \frac{e^r}{ (1+\frac{1}{4} (e^r +t)^2)^m} +\frac{e^r}{(1+\frac{1}{4} (e^r-t)^2)^m} dr\\
&=\frac{L_{m-1}^{0, \mathbf{e}_1}(t)}{\sqrt{1+\alpha^2}}=\frac{L_{m-1}^{0, \mathbf{e}_1}(0)}{\sqrt{1+\alpha^2}}
\end{align*}
Where here the final equality follows from the observation that last line is exactly what would be obtained from evaluating on $S_0$, i.e.,  the $x$-axis, and with $\mathbf{x}_0=\mathbf{e}_1=(1,0)$.  In particular, the value is independent of $t$.  We conclude that for all $m\geq 0$, 
$$
\bar{\lambda}_{LY}^{1+2m}[S_\alpha \times \Real^{2m}]\leq \sqrt{1+\alpha^2} \bar{\lambda}_{LY}^{1+2m}[S_0 \times \Real^{2m}]= \sqrt{1+\alpha^2}.
$$
Combined with the lower bound already established this verifies that
$$
 \sqrt{1+\alpha^2}=\bar{\lambda}_{LY}^\infty[S_\alpha]=\bar{\lambda}_{LY}^1[S_\alpha].
 $$
 
Finally, by Proposition \ref{ConfVolCMentProp} and the previously established lower bound
$$
\sqrt{1+\alpha^2}\leq \lambda_{CM}^1[S_\alpha]\leq \bar{\lambda}_{LY}^\infty[S_\alpha]= \sqrt{1+\alpha^2},
$$
which completes the proof.
\end{proof}

\section{Further Questions and Observations}
First, we ask whether the inequalities between Colding-Minicozzi entropy, stable conformal volume and virtual entropy are ever strict when $\Sigma$ is not a self-shrinker.
\begin{ques}
	Is there a $n$-dimensional submanifold, possibly singular,  $\Sigma \subset \Real^N$, with $\lambda_{CM}^n[\Sigma]<\lambda_V^n[\Sigma]$?  More generally, can one ever have $\lambda_{CM}^n[\Sigma]<\bar{\lambda}_{LY}^\infty[\Sigma]<\lambda_V^n[\Sigma]$? If $\lambda_{CM}^n[\Sigma]<\bar{\lambda}_{LY}^\infty[\Sigma]$ can occur,  may one also have $\lambda_{CM}^n[\Sigma]<\bar{\lambda}_{LY}^n[\Sigma]$?
\end{ques} 
	It would be particularly interesting to compute $\lambda_{V}^1[S_\alpha]$ for  the $S_{\alpha}$ from Appendix \ref{LoxodromeSec}.

\begin{ques}
	What is the lowest entropy self-shrinking $\mathbb{RP}^2$ in $\Real^N$?
\end{ques}
 When $N=3$, any immersed $\mathbb{RP}^2$ must have a triple point \cite{banchoffTriplePointsSingularities1974}, and so entropy at least three.  There are no candidate self-shrinkers in this setting -- see however \cite{whiteBoundarySingularitiesMean2022} where the existence of a shrinking M\"{o}bius band bounded by a straight line in $\Real^3$ is shown. For $N\geq 5$, it follows from \cite[Theorem 1.1]{andrewsStabilitySelfShrinkingSolutions2014} that $\Sigma_{\mathbb{RP}^2}$ is (Colding-Minicozzi) entropy unstable, suggesting there should be a lower entropy unoriented shrinker.  Likewise, as $ \Sigma_{\mathbb{RP}^2}$ is unstable as a minimal surface in $\mathbb{S}^4$, there is a non-constant flow out of it that remains in the sphere.  As observed in \cite{zhuGeometricVariationalProblems2018} such a flow has decreasing Colding-Minicozzi entropy when viewed in the ambient $\Real^5$.  As the flow is constrained to lie in $\mathbb{S}^4$ and must become singular, this suggests the existence of an unoriented self-shrinker of lower entropy in $\mathbb{R}^4$.  However, the fact that $\lambda^2_{CM}[\Sigma_{\mathbb{RP}^2}]>2$ introduces significant technical issues.

In their proof of the Willmore conjecture \cite{MarquesNeves}, Marques-Neves showed that any positive genus surface in $\mathbb{S}^3$ sits inside a five parameter family and at least one member of the family has area at least that of the Clifford torus.  Together with Theorem \ref{LYCMIneqThm},  this suggests:
\begin{conj}
	Let $\Sigma\subset \Real^3$ be a closed self-shrinker of positive genus, then
	$$
	1.57 \approx \frac{\pi}{2} =\bar{\lambda}_{LY}^2[\Sigma_C]\leq    \lambda_{CM}^2[\Sigma] 
	$$
	where here $\Sigma_C\subset \mathbb{S}^3$ is the Clifford torus.
\end{conj}
It was shown in \cite{BernsteinWang2} that such a $\Sigma$ satisfies the bound
$$
1.52 \approx \lambda_{CM}^2[\mathbb{S}^1\times \Real]\leq \lambda_{CM}^2[\Sigma]. 
$$
In \cite{AngenentEntropy}, numerical methods established that there a rotationally symmetric genus one self-shrinker, i.e., an Angenent torus \cite{Angenent}, $\Sigma_A\subset \Real^3$,  whose entropy satisfies 
$$
\lambda_{CM}^2[\Sigma_A]\approx 1.852.
$$
Recently,  Chu-Sun \cite{ChuSun} have shown that no Angenent torus can be the lowest entropy positive genus shrinker.  Their work suggests that the lowest entropy is achieved by a non-compact shrinker of genus one.

Finally, by using the heat kernel on hyperbolic space, a notion of Colding-Minicozzi entropy has been extended to some spaces of non-positive curvature \cite{bernsteinColdingMinicozziEntropy2020, BernsteinBhattacharyaCartan, bernsteinMinimalSurfacesColdingMinicozzi2024, zhuGeometricVariationalProblems2018}. 
\begin{ques}
	Is there a useful notion of virtual entropy in other ambient Riemannian manifolds -- in particular those of non-positive curvature?
\end{ques}

\bibliographystyle{hamsabbrv}
\bibliography{Library}

\providecommand{\bysame}{\leavevmode\hbox to3em{\hrulefill}\thinspace}
\providecommand{\MR}{\relax\ifhmode\unskip\space\fi MR }
\providecommand{\MRhref}[2]{%
  \href{http://www.ams.org/mathscinet-getitem?mr=#1}{#2}
}
\providecommand{\href}[2]{#2}
\begin{thebibliography}{10}

\bibitem{andrewsStabilitySelfShrinkingSolutions2014}
B.~Andrews, H.~Li, and Y.~Wei, \emph{{$\mathcal{F}$}-{S}tability for
  {Self-Shrinking Solutions} to {Mean Curvature Flow}}, Asian Journal of
  Mathematics \textbf{18} (2014), no.~5, 757--778.

\bibitem{Angenent}
S.~B. Angenent, \emph{Shrinking doughnuts}, Nonlinear Diffusion Equations and
  Their Equilibrium States, 3 ({{Gregynog}}, 1989), Progr. {{Nonlinear}}
  Differential Equations Appl., vol.~7, Birkh{\"a}user Boston, Boston, MA,
  1992, pp.~21--38.

\bibitem{banchoffTriplePointsSingularities1974}
T.~Banchoff, \emph{Triple {{Points}} and {{Singularities}} of {{Projections}}
  of {{Smoothly Immersed Surfaces}}}, Transactions of the American Mathematical
  Society \textbf{46} (1974), no.~3, 402--406.

\bibitem{AngenentEntropy}
Y.~I. {Berchenko-Kogan}, \emph{The entropy of the {{Angenent}} torus is
  approximately 1.85122}, Journal of Experimental Mathematics \textbf{30}
  (2021), no.~4, 587--594.

\bibitem{bernsteinColdingMinicozziEntropy2020}
J.~Bernstein, \emph{Colding {{Minicozzi Entropy}} in {{Hyperbolic Space}}},
  Nonlinear Anal. \textbf{210} (2020), 112401.

\bibitem{BernsteinRigidity}
\bysame, \emph{Rigidity properties of {{C}}olding-{{M}}inicozzi entropies},
  Advanced Nonlinear Studies \textbf{24} (2024), no.~1, 155--166.

\bibitem{BernsteinBhattacharyaCartan}
J.~Bernstein and A.~Bhattacharya, \emph{{C}olding-{M}inicozzi entropies in
  {C}artan-{H}adamard manifolds}, 2022.

\bibitem{bernsteinMinimalSurfacesColdingMinicozzi2024}
\bysame, \emph{Minimal surfaces and {{Colding-Minicozzi}} entropy in complex
  hyperbolic space}, Geometriae Dedicata \textbf{218} (2024), Article 61.

\bibitem{BernsteinWang1}
J.~Bernstein and L.~Wang, \emph{A sharp lower bound for the entropy of closed
  hypersurfaces up to dimension six}, Invent. Math. \textbf{206} (2016), no.~3,
  601--627.

\bibitem{BernsteinWang2}
\bysame, \emph{A topological property of asymptotically conical self-shrinkers
  of small entropy}, Duke Math J. \textbf{166} (2017), no.~3, 403--435.

\bibitem{BernsteinWang3}
\bysame, \emph{Topology of closed hypersurfaces of small entropy}, Geom. Topol.
  \textbf{22} (2018), no.~2, 1109--1141.

\bibitem{BWIsotopy}
\bysame, \emph{Closed hypersurfaces of low entropy in {{R}}{$^4$} are
  isotopically trivial}, Duke Math. J. \textbf{171} (2022), no.~7, 1531--1558.

\bibitem{BWDensity}
\bysame, \emph{Lower bounds on density for topologically nontrivial minimal
  cones up to dimension six},  (2024).

\bibitem{bryantSurfacesConformalGeometry1988}
R.~L. Bryant, \emph{Surfaces in conformal geometry}, Proceedings of
  {{Symposia}} in {{Pure Mathematics}} (R.~Wells, ed.), vol.~48, {American
  Mathematical Society}, {Providence, Rhode Island}, 1988, pp.~227--240.

\bibitem{bryantConformalVolume2tori2015}
\bysame, \emph{On the conformal volume of 2-tori}, July 2015.

\bibitem{caoMatrixLiYauHamiltonEstimates2005}
H.-D. Cao and L.~Ni, \emph{Matrix {{Li-Yau-Hamilton}} estimates for the heat
  equation on {{K}}\"{a}hler manifolds}, Math. Ann. \textbf{331} (2005), no.~4,
  795--807.

\bibitem{chenRigidityStabilitySubmanifolds2021}
L.~Chen, \emph{Rigidity and stability of submanifolds with entropy close to
  one}, Geom Dedicata \textbf{215} (2021), no.~1, 133--145.

\bibitem{CCMS}
O.~Chodosh, K.~Choi, C.~Mantoulidis, and F.~Schulze, \emph{Mean curvature flow
  with generic low-entropy initial data}, 2021.

\bibitem{CMS}
O.~Chodosh, C.~Mantoulidis, and F.~Schulze, \emph{Mean curvature flow with
  generic low-entropy initial data {II}}, 2023.

\bibitem{ChuSun}
A.~C.-P. Chu and A.~Sun, \emph{Genus one singularities in mean curvature flow},
  2023.

\bibitem{CIMW}
T.~H. Colding, T.~Ilmanen, W.~P. Minicozzi~II, and B.~White, \emph{The round
  sphere minimizes entropy among closed self-shrinkers}, J. Differential Geom.
  \textbf{95} (2013), no.~1, 53--69.

\bibitem{Coldinga}
T.~H. Colding and W.~P. Minicozzi~II, \emph{Generic mean curvature flow {{I}};
  generic singularities}, Ann. of Math. (2) \textbf{175} (2012), no.~2,
  755--833.

\bibitem{coldingEntropyCodimensionBounds2019}
\bysame, \emph{Entropy and codimension bounds for generic singularities}, July
  2019.

\bibitem{coldingComplexityParabolicSystems2020}
\bysame, \emph{Complexity of parabolic systems}, Publ.math.IHES \textbf{132}
  (2020), no.~1, 83--135.

\bibitem{Dodziuk}
{Dodziuk}, \emph{Maximum principle for parabolic inequalities and the heat flow
  on open manifolds}, Indiana Univ. Math. J. \textbf{32} (1983), no.~5,
  703--716.

\bibitem{donnellyUniquenessPositiveSolutions1987}
H.~Donnelly, \emph{Uniqueness of positive solutions of the heat equation},
  Proceedings of the American Mathematical Society \textbf{99} (1987), no.~2,
  353--356.

\bibitem{Ecker2001}
K.~Ecker, \emph{Lectures on {{Regularity}} for {{Mean Curvature Flow}}}, vol.
  154, 2001.

\bibitem{gromovFillingRiemannianManifolds1983}
M.~Gromov, \emph{Filling {{Riemannian}} manifolds}, J. Differential Geom.
  \textbf{18} (1983), no.~1, 1--147.

\bibitem{hamiltonMatricHarnackEstimate1993}
R.~S. Hamilton, \emph{A {{Matric Harnack Estimate}} for the {{Heat Equation}}},
  Communications in Analysis and Geometry \textbf{1} (1993), no.~1, 113--126.

\bibitem{hamiltonMonotonicityFormulasParabolic1993}
\bysame, \emph{Monotonicity formulas for parabolic flows on manifolds},
  Communications in Analysis and Geometry \textbf{1} (1993), no.~1, 127--137.

\bibitem{HuiskenMon}
G.~Huisken, \emph{Asymptotic behaviour for singularities of the mean curvature
  flow}, J. Differential Geom. \textbf{31} (1990), no.~1, 285--299.

\bibitem{liNewConformalInvariant1982a}
P.~Li and S.-T. Yau, \emph{A new conformal invariant and its applications to
  the {{Willmore}} conjecture and the first eigenvalue of compact surfaces},
  Invent Math \textbf{69} (1982), no.~2, 269--291.

\bibitem{liParabolicKernelSchr6dinger1986}
\bysame, \emph{On the parabolic kernel of the {{S}}chr\"{o}dinger operator},
  Acta Math. \textbf{156} (1986), 53--201.

\bibitem{magman08}
A.~Magni and C.~Mantegazza, \emph{Some remarks on {H}uisken's monotonicity
  formula for mean curvature flow}, in ``Singularities in Nonlinear Evolution
  Phenomena and Applications" (M. Novaga and G. Orlandi eds.), CRM Series of
  Center ``Ennio De Giorgi", Pisa (2009), 157--169.

\bibitem{MarquesNeves}
F.~C. Marques and A.~Neves, \emph{Min-max theory and the {{Willmore}}
  conjecture}, Ann. of Math. (2) \textbf{179} (2014), no.~2, 683--782.

\bibitem{montielMinimalImmersionsSurfaces1986}
S.~Montiel and A.~Ros, \emph{Minimal immersions of surfaces by the first
  {{Eigenfunctions}} and conformal area}, Invent Math \textbf{83} (1986),
  no.~1, 153--166.

\bibitem{vazquezAsymptoticBehaviourMethods2018}
J.~L. V{\'a}zquez, \emph{Asymptotic behaviour methods for the {{Heat
  Equation}}. {{Convergence}} to the {{Gaussian}}}, November 2018.

\bibitem{whiteBoundarySingularitiesMean2022}
B.~White, \emph{Boundary singularities in mean curvature flow and total
  curvature of minimal surface boundaries}, Comment. Math. Helv. \textbf{97}
  (2022), no.~4, 669--689.

\bibitem{zhuGeometricVariationalProblems2018}
J.~Zhu, \emph{Geometric {{Variational Problems}} for {{Mean Curvature}}}, Ph.D.
  thesis, Harvard, Cambridge, MA, 2018.

\bibitem{JZhu}
\bysame, \emph{On the entropy of closed hypersurfaces and singular
  self-shrinkers}, J. Differential Geom. \textbf{114} (2020), no.~3, 551--593.

\end{thebibliography}
\end{document}